\documentclass[10pt]{article}

\usepackage[utf8]{inputenc}
\usepackage[T1]{fontenc}

\usepackage{epsf}
\usepackage{amsmath}

\allowdisplaybreaks

\usepackage[showframe=false]{geometry}
\usepackage{changepage}

\usepackage{epsfig}
\usepackage{amssymb}

\usepackage{amsthm}
\usepackage{setspace}
\usepackage{cite}
\usepackage{mcite}

\usepackage{algorithmic}  
\usepackage{algorithm}

\usepackage{shadow}
\usepackage{fancybox}
\usepackage{fancyhdr}

\usepackage{color}
\usepackage[usenames,dvipsnames,svgnames,table]{xcolor}

\definecolor{mypurple}{rgb}{.4,.0,.5}

\usepackage[hyphens]{url}

\usepackage[colorlinks=true,
            linkcolor=black,
            urlcolor=blue,
            citecolor=purple]{hyperref}

\usepackage{breakurl}

\def\s{{\bf s}}

\def\x{{\bf x}}

\def\x{{\mathbf x}}

\def\s{{\bf s}}

\def\x{{\bf x}}

\def\tr{\mbox{Tr}}

\def\tr{{\rm tr}\,}

\def\be{\begin{equation}}
\def\ee{\end{equation}}
\def\ba{\left[\begin{array}}
\def\ea{\end{array}\right]}

\def\s{{\bf s}}

\def\x{{\bf x}}

\def\1{{\bf 1}}

\def\g{{\bf g}}
\def\0{{\bf 0}}

\def\cG{{\mathcal G}}

\def\mR{{\mathbb R}}

\def\mN{{\mathbb N}}
\def\mE{{\mathbb E}}
\def\mS{{\mathbb S}}

\def\lp{\left (}
\def\rp{\right )}

\def\s{{\bf s}}

\def\x{{\bf x}}

\def\x{{\mathbf x}}

\def\s{{\bf s}}

\def\x{{\bf x}}

\def\tr{\mbox{Tr}}

\def\tr{{\rm tr}\,}

\def\be{\begin{equation}}
\def\ee{\end{equation}}
\def\ba{\left[\begin{array}}
\def\ea{\end{array}\right]}

\def\s{{\bf s}}

\def\x{{\bf x}}

\def\({\left (}
\def\){\right )}

\def\1{{\bf 1}}

\def\g{{\bf g}}
\def\0{{\bf 0}}

\def\cX{{\mathcal X}}

\def\cL{{\mathcal L}}

\def\cX{{\mathcal X}}
\def\cQ{{\mathcal Q}}

\usepackage{xcolor}
\usepackage{color}

\definecolor{darkgreen}{rgb}{0, 0.4,0}

\definecolor{purplebrown}{rgb}{0.5,0.1,0.6}

\definecolor{ultclupcol}{rgb}{0.1,0.5,0.5}

\definecolor{mytrycolor}{rgb}{0.5,0.7,0.2}

\definecolor{ultclupcola}{rgb}{.5,0,.5}

\definecolor{shadebrown}{rgb}{0.1,0.1,0.9}
\definecolor{lightblue}{rgb}{0.2,0,1}

\usepackage{fancybox}
\usepackage{graphicx}
\usepackage{epstopdf}
\usepackage{epsfig}
\usepackage{wrapfig}
\usepackage{subfigure}

\usepackage{xcolor}
\usepackage{tcolorbox}

\newtcbox{\xmybox}{on line,
arc=7pt,
before upper={\rule[-3pt]{0pt}{10pt}},boxrule=0pt,
boxsep=0pt,left=6pt,right=6pt,top=0pt,bottom=0pt,enhanced, coltext=blue, colback=white!10!yellow}

\newtcbox{\xmyboxa}{on line,
arc=7pt,
before upper={\rule[-3pt]{0pt}{10pt}},boxrule=0pt,
boxsep=0pt,left=6pt,right=6pt,top=0pt,bottom=0pt,enhanced, colback=white!10!yellow}

\newtcbox{\xmyboxb}{on line,
arc=7pt,
before upper={\rule[-3pt]{0pt}{10pt}},boxrule=1pt,colframe=darkgreen!100!blue,
boxsep=0pt,left=6pt,right=6pt,top=0pt,bottom=0pt,enhanced, colback=white!10!yellow}

\newtcbox{\xmyboxc}{on line,
arc=7pt,
before upper={\rule[-3pt]{0pt}{10pt}},boxrule=.7pt,colframe=blue!100!blue,
boxsep=0pt,left=6pt,right=6pt,top=0pt,bottom=0pt,enhanced, coltext=blue, colback=white!10!yellow}

\newtcbox{\xmytboxa}{on line,
arc=7pt,
before upper={\rule[-3pt]{0pt}{10pt}},boxrule=.0pt,colframe=pink!50!yellow,
boxsep=0pt,left=6pt,right=6pt,top=0pt,bottom=0pt,enhanced, coltext=white, colback=blue!40!red}

\newtcbox{\xmytboxb}{on line,
arc=7pt,
before upper={\rule[-3pt]{0pt}{10pt}},boxrule=.0pt,colframe=pink!50!yellow,
boxsep=0pt,left=6pt,right=6pt,top=0pt,bottom=0pt,enhanced, coltext=white, colback=white!40!green}

\makeatletter
\newcommand\subsubsubsection{\@startsection{paragraph}{4}{\z@}{-2.5ex\@plus -1ex \@minus -.25ex}{1.25ex \@plus .25ex}{\normalfont\normalsize\bfseries}}
\newcommand\subsubsubsubsection{\@startsection{subparagraph}{5}{\z@}{-2.5ex\@plus -1ex \@minus -.25ex}{1.25ex \@plus .25ex}{\normalfont\normalsize\bfseries}}
\makeatother

\newtheorem{theorem}{Theorem}

\newtheorem{corollary}{Corollary}

\newtheorem{remark}{Remark}

\begin{document}

\begin{singlespace}

\title {An RDT based confirmation of Lehner's formula for Kronecker-Gaussian matrices
}
\author{
\textsc{Mihailo Stojnic
\footnote{e-mail: {\tt flatoyer@gmail.com}}
}}
\date{}
\maketitle

\centerline{{\bf Abstract}} \vspace*{0.1in}

Remarkable breakthroughs \cite{Schultz05,HaagThor05} established the so-called strong asymptotic freeness between classical Gaussian random ensembles and their semicircular free counterparts. Along the same lines, the Lehner formula \cite{Lehner99}, associated with the free counterpart, precisely determines the spectral edges of Kronecker-Gaussian matrices.

We here revisit and study this formula without the utilization of random matrix theory and spectral methods. In particular, relying on concepts utilized within \emph{Random Duality Theory} (RDT) \cite{StojnicCSetam09,StojnicRegRndDlt10,Stojniccovmat26}, we reconfirm Lehner's formula and effectively reprove the key asymptotic freeness results obtained via spectral methods in \cite{HaagThor05,Schultz05}.

\vspace*{0.25in} \noindent {\bf Index Terms: Strong asymptotic freeness; Lehner's formula; Random duality theory (RDT)}.

\end{singlespace}

\section{Introduction}
\label{sec:back}

Excellent progress has been achieved over the last couple of decades in studying the so-called free probability equivalent models. Remarkable breakthroughs \cite{Schultz05, HaagThor05} considered the norm of polynomials of classical Gaussian random ensembles and showed a limiting convergence towards their semicircular free counterparts, effectively establishing the so-called strong \emph{asymptotic} freeness. Many extensions followed, including different types of randomness \cite{Anderson13} or even combinations of deterministic and random (Wigner/GUE/Haar/permutations) scenarios \cite{Parraud22, CollMale14, Male12, GuionnShly09, CollinsGP22, Belinschi17}. The range of applications turned out to be diverse as well and has recently expanded rapidly to encompass, among others, geometric \cite{HideMagee23}, graph-related \cite{LubPer16, BorCollins19}, and algebraic considerations \cite{Hayes22}—many of which allowed to settle or resolve long-standing open problems and conjectures \cite{Hayes22, HaagThor05, LubPer16, BorCollins19}.

The recent establishing of the so-called \emph{intrinsic} freeness \cite{bandfree23, bandfree24} has been particularly fruitful. It allowed for a way to circumvent the asymptotic scenarios and to effectively analyze spectral properties of almost any Gaussian matrix. This has a plethora of applications, including precise studying algorithmic and information-theoretic properties of classical spiked models and sample covariances  \cite{bandfree23, bandfree24, HandelSurvey26}.

In \cite{Lehner99}, Lehner developed an elegant deterministic formula for the spectral edges characterization of semicircular free equivalent to Kronecker–Gaussian ensembles. In this work, we deviate from random matrix theory and develop a different methodology to establish Lehner's formula  without relying on the associated strong asymptotic freeness. In particular, we show that leaning on concepts utilized within Random Duality Theory (RDT) \cite{StojnicCSetam09, StojnicRegRndDlt10, Stojniccovmat26} is sufficient to reprove the key asymptotic freeness results established through spectral characterizations in \cite{HaagThor05, Schultz05}.

\section{Kronecker--Gaussian matrices -- preliminaries and prior work}
\label{sec:sampcov}

To put everything on right mathematical track, we first introduce precise definitions of objects that we will study.

\subsection{Mathematical setup and contextualization}
\label{sec:context}

 Let $n\in\mN$, $k\in\mN$, and  $l\in\mN$ be three positive integers. We operate in a large dimensional context and assume large $n$. While our analysis allows that $k$ and $l$ depend on $n$, for the easiness of presentation we take them as fixed. 

Let $\{\bar{G}_i=\bar{G}_i^T\in\mR^{n\times n}\}_{i=1,\dots,l}$ be a collection of independent symmetric standard normal Gaussian matrices. In particular, for $G_i\in\mR^{n\times n}$ having independent standard normal components, we set
\begin{equation}\label{eq:amat1a0}
 \bar{G}_i = \frac{1}{\sqrt{2}} \lp G_i + G_i^T\rp, i=1\dots,l. 
 \end{equation}
Let $\{A_i=A_i^T\in\mR^{k\times k}\}_{i=0,1,\dots,l}$ be a collection of (symmetric) positive semi-definite deterministic matrices. Assuming that $I\in\mR^{n\times n}$ is the standard identity matrix, the following sum of Kroneckered Gaussians plays a critical role in studying spectral norms of matrix polynomials \cite{HaagThor05,CollinsGP22,Schultz05}
\begin{equation}\label{eq:amat1a0b0}
 H = A_0\otimes I + \frac{1}{\sqrt{n}} \sum_{i=1}^{l} A_i\otimes \bar{G}_i. 
 \end{equation}
A simplified main takeaway is that the spectrum of $H$ is (asymptotically in $n$) close to the spectrum of the following semicircular free equivalent
\begin{equation}\label{eq:amat1a0b0}
 H_{free} = A_0\otimes I + \frac{1}{\sqrt{n}} \sum_{i=1}^{l} A_i\otimes \s_i, 
 \end{equation}
where $\s_i$ are from a semicircular family.

 Let function $\lambda_i(\cdot)$ determine the $i$-th smallest eignevalue of its argument. One then  has
 \begin{equation}\label{eq:amat1a0b1}
 \lambda_1(H) \leq  \lambda_2(H) \leq \dots  \leq  \lambda_n(H) , 
 \end{equation}
 where $ \lambda_1(H)$ and  $ \lambda_n(H)$ are the smallest and the largest eigenvalues of $H$, respectively. Consider the following Lehner spectral edge characterization \cite{Lehner99} 
\begin{eqnarray}\label{eq:amat1a0b2}
 \rho_n &  \triangleq & \inf_{Z\succeq 0 } \lambda_n\lp A_0\otimes I + Z+ \frac{1}{\sqrt{n}} \sum_{i=1}^{l} A_iZ^{-1}A_i\rp 
 \nonumber \\
 \rho_1 & \triangleq & \sup_{0\succeq Z } \lambda_1\lp A_0\otimes I + Z+ \frac{1}{\sqrt{n}} \sum_{i=1}^{l} A_iZ^{-1}A_i\rp  . 
 \end{eqnarray}
One notes that objects in (\ref{eq:amat1a0b2}) are fully deterministic. Moreover, relying on Schur complement positive definiteness, the underlying optimizations can readily be transformed into a convex SDP. (As $k$ increases, SDPs may not necessarily be among the programs most amenable to practical implementations; nonetheless, they are at least in principle solvable in polynomial time. For more on the optimization aspects of (\ref{eq:amat1a0b2}), see \cite{Kunisky26}). We also observe that $\inf$ is used even though the formulations are finite-dimensional. To make writing easier, in the rest of the paper, we assume that $A_i$'s are such that $\inf$ can be replaced by $\max$.

Various forms of association between $\rho_1$ and $\rho_n$ on the one side and $\lambda_1(H)$ and $\lambda_n(H)$ on the other have been proven in \cite{ChenGVTH26,Anderson13,bandfree23,bandfree24,HaagThor05,Schultz05,CollinsGP22} through a host of different approaches. The employed methodologies range from classical free probability Stieltjes transform/matrix Dyson equation \cite{HaagThor05,Schultz05,CollinsGP22} and first order \cite{BorCollins19,BordColl24,Anderson13} based methods, to matrix concentrations and interpolations \cite{CollinsGP22,bandfree23,bandfree24,ChenGVTH26}. (For studies particularly tailored for unitary and/or permutations ensembles, see \cite{BorCollins19,Magee25,Cassidy24,BordColl24,CollMale14}, respectively; for more general ensembles, including those unrelated to Gaussian, see, e.g., \cite{GuionnShly09,BrailUniv24}). Following the above-mentioned spectrum closeness main takeaway (relevant to the dimensional scenario of our interest here), one has
 \begin{eqnarray}\label{eq:amat1a0b3}
\lim_{n\rightarrow \infty} \lambda_1(H) = \rho_1 \quad \quad \mbox{and} \quad \quad
\lim_{n\rightarrow \infty} \lambda_n(H) = \rho_n , 
 \end{eqnarray}
 with very strong concentrations as well. In the regimes of our interest here, concentrations are induced by large $n$. However, that is not necessarily needed and, in general, the results of \cite{ChenGVTH26,Anderson13,bandfree23,bandfree24,HaagThor05,Schultz05,CollinsGP22} are actually stronger and cover a variety of different dimensional scenarios.

We would particularly highlight \cite{bandfree23,bandfree24}, which establishes the so-called \emph{intrinsic} freeness. Differently from, for example, \cite{HaagThor05,CollinsGP22,Schultz05}, which operate in an asymptotic $n\rightarrow \infty$ regime, \cite{bandfree23,bandfree24} take $n=1$ and manage to establish that concentrations can be deduced via $\|\mbox{cov}(H)\|$ and $k$. While the $n=1$ choice might appear as if one deviates from the block structure, it actually allows in a convenient fashion considerations of a variety of options for covariance of $H$ that are often very relevant in practical consideration. It is particularly useful in scenarios where matrices deviate from standard normals, but in a highly unpredictable way. While a host of problems that can be handled through these methods is discussed in \cite{bandfree24}, we would here single out applications related to algorithmic properties of classical spiked models and sample covariances. Namely, the results of \cite{bandfree23,bandfree24} allow handling these problems on a very precise level and accurately determine key associated features, such as residual estimation errors and BBP phase transitions, which is a stark contrast with typical scaling-order characterizations \cite{KolLou17,Tropp11,VershNonAsym12,KanLovSim97, Bourgain99, Rudelson99, GianHarTso05, Paouris06, Adam10,Adamczak11} (for more on alternative precise analyses methodologies, see also, e.g., \cite{BBP05,MontRich14,DonGavJohn18,Lesetal17,PerryWB20,BarbMM17,PourBM24,BehneReeves22,PakKK23,Stojniccovmat26}).

A different approach towards establishing (\ref{eq:amat1a0b3}) was discussed in \cite{CollYama26}. Utilizing the Sudakov-Fernique \cite{Sudakov71,Fernique74,Fernique75,Vitale00} comparison technique and assuming positive semi-definiteness of $A_i$, \cite{CollYama26} showed that slightly (negligibly) modified $\rho_1$ and $\rho_n$ are the lower and upper bounds on $\lim_{n\rightarrow \infty} \mE \lambda_1(H) $ and $\lim_{n\rightarrow \infty} \mE \lambda_n(H) $.

We here prove (\ref{eq:amat1a0b3}) relying on comparison concepts as well. However, we lean on methodologies utilized within \emph{Random Duality Theory} (RDT) \cite{StojnicCSetam09,StojnicRegRndDlt10,Stojniccovmat26} and prove that $\rho_n$ is not only the upper bound on $\lim_{n\rightarrow \infty} \mE\lambda_n(H) $, but also its exact value (the same then automatically holds for $\rho_1$ and $\lim_{n\rightarrow \infty} \mE \lambda_1(H) $).

\subsection{Replacing Kronecker--Gaussians with plain Gaussians}
\label{sec:kronrep}
 
Before we start with the RDT machinery, we find it useful to introduce a few technical preliminaries. First, we note that once $\lambda_n(H) $ is handled, $\lambda_1(H)$ follows immediately by flipping a sign and relying on Gaussian symmetry (i.e., on the fact that $\bar{G}_i$ and $-\bar{G}_i$ are identically distributed). In the rest of the paper, we therefore focus only on $\lambda_n(H)$. 

We start by observing that
\begin{equation}\label{eq:kronrep1}
\lambda_n(H)  = \max_{\x\in\mS^{kn}} \x^T\lp A_0\otimes I + \frac{1}{\sqrt{n}} \sum_{i=1}^{l} A_i\otimes \bar{G}_i \rp \x, 
 \end{equation}
where $\mS^{nk}$ stands for the unit sphere in $\mR^{nk}$, i.e., $\mS^{nk}=\{\x\in\mR^{nk} \hspace{.03in} | \hspace{.03in} \|\x\|_2=1 \}$. Let $\bar{X}\in\mR^{n\times k}$ be a $n\times k$ matrix such that 
\begin{equation}\label{eq:kronrep2}
\x=\mbox{vec}(\bar{X}), 
\end{equation}
where $\mbox{vec}(\cdot)$ stacks the columns (starting from the first to the last) of its argument into a vector. Also, let $\bar{X}_i,i=1\dots,k$ be the $i$-th column of $\bar{X}$. One then has
\begin{equation}\label{eq:kronrep3}
\|\x\|_2^2=\|\mbox{vec}(\bar{X})\|_2^2 = \sum_{i=1}^{k}\|\bar{X}_i\|_2^2 = \tr(\bar{X}^T\bar{X}). 
\end{equation}
Moreover,
\begin{eqnarray}\label{eq:kronrep4}
\lambda_n(H) &  = &  \max_{\tr(\bar{X}^T\bar{X})=1} \lp \sum_{r=1}^n\sum_{s=1}^k \lp A_0\rp_{rs}\bar{X}_r^T\bar{X}_s + \frac{1}{\sqrt{n}} \sum_{i=1}^{l} \sum_{r=1}^n\sum_{s=1}^k \lp A_i\rp_{rs}\bar{X}_r^T \bar{G}_i \bar{X}_s \rp 
\nonumber \\
\lambda_n(H) & =  & \max_{\tr(\bar{X}^T\bar{X})=1} \lp \tr(A_0 \bar{X}^T\bar{X}) + \frac{1}{\sqrt{n}} \sum_{i=1}^{l} \tr(A_i \bar{X}^T \bar{G}_i \bar{X} ) \rp 
\nonumber \\
  & =  & \max_{\tr(\bar{X}^T\bar{X})=1} \lp \tr(A_0 \bar{X}^T\bar{X}) + \frac{\sqrt{2}}{\sqrt{n}} \sum_{i=1}^{l} \tr(A_i \bar{X}^T G_i \bar{X} ) \rp ,
   \end{eqnarray}
where, for $i=0,1,\dots,l$,  the element of $A_i$ in the intersection of the $r$-th row and the $s$-th column is $\lp A_i\rp_{rs}$. The above effectively replaces the Kronecker--Gaussians with plain Gaussians.  In what follows, we characterize  $\lambda_n(H)$, via  Random Duality Theory (RDT).

\section{RDT characterization of $\lambda_n(H)$}
\label{sec:xirdt}

We start by recalling on the following  four key RDT principles \cite{StojnicCSetam09,StojnicRegRndDlt10,Stojniccovmat26} (for more on  upgrades and associated algorithmic implications, see, e.g., \cite{Stojnicalgbp25,Stojnicclupsk25}):

\begin{enumerate}

\item \emph{Finding underlying optimization algebraic representation}

\item \emph{Determining random dual} 

\item \emph{Handling random dual}

\item \emph{Double-checking strong random duality.}

\end{enumerate}
Each of these four principles, together with its relation to the problem of interest here, is discussed in detail below. Before proceeding with the details, we should also mention that the first two steps are sufficient to obtain upper bounds and have been discussed in \cite{CollYama26}. Here we include these steps for the completeness and discuss them through a different comparison concept. The third step is fairly simple and is included to ensure complete matching and practical recovery of prior results. Finally, as is usual the case within the RDT, the most challenging is the last step and it is typically the heart of the entire machinery.

\subsection{Finding underlying optimization algebraic representation} 
\label{sec:randpr}

We start with an ortho-diagonal transformation that makes writings below slightly neater. To that end we observe that for $X\in\mR^{n\times k} $ and  $R\in\mR^{k\times k} $ such that $X^TX=I$ and $\tr(R^TR)=1$,  change of variables $\bar{X} = XR$ gives  
\begin{equation}\label{eq:inteq1ab0}
\bar{X}^T\bar{X}=R^TR \quad \mbox{and}\quad \frac{\sqrt{2}}{\sqrt{n}} \tr(A_i \bar{X}^T G_i \bar{X} ) =     \frac{\sqrt{2}}{\sqrt{n}} \tr(A_i R^TX^T G_i XR ).
\end{equation}
We set
\begin{equation}\label{eq:inteq1ab0b0}
F_i = F_i(R) \triangleq RA_iR^T,
\end{equation}
and rewrite (\ref{eq:kronrep4}) as
\begin{eqnarray}\label{eq:kronrep5}
\lambda_n(H) & =  & \max_{X^TX=I,\tr(R^TR)=1} \lp \tr(A_0R^TR) +  \frac{\sqrt{2}}{\sqrt{n}} \sum_{i=1}^{l}  \tr(A_i R^TX^T G_i XR ) \rp 
\nonumber \\
\lambda_n(H) & =  & \max_{\tr(R^TR)=1} \lp \tr(A_0R^TR) +  \max_{X^TX=I}  \frac{\sqrt{2}}{\sqrt{n}} \sum_{i=1}^{l}  \tr(A_i R^TX^T G_i XR ) \rp 
\nonumber \\
\lambda_n(H) & =  & \max_{\tr(R^TR)=1} \lp \tr(F_0) +  \max_{X^TX=I}  \frac{\sqrt{2}}{\sqrt{n}} \sum_{i=1}^{l}  \tr(X^T G_i XF_i ) \rp 
\nonumber \\
\lambda_n(H) & =  & \max_{\tr(R^TR)=1} \lp \tr(F_0) + \frac{\sqrt{2}}{\sqrt{n}} \xi \rp  ,
\end{eqnarray}
where
\begin{eqnarray}\label{eq:kronrep6}
 \xi  & \triangleq &  \max_{X^TX=I}   \sum_{i=1}^{l}  \tr(X^T G_i XF_i ) ,
\end{eqnarray}
 is called \emph{random primal} within the RDT.

\subsection{Determining random dual} 
\label{sec:randdual}

The following theorem establishes the so-called random dual.

\begin{theorem}
\label{thm:thm1}
For large $n\in\mN$ and fixed $k,l\in\mN$, let $G_{i}\in\mR^{n\times n}$ and $G_{i}^{(1)}\in\mR^{k\times n},i=1,\dots,l$ be independent  standard normal matrices. For deterministic $A_i\in\mR^{k\times k}$ such that $A_i=A_i^T\succeq 0,i=1,\dots,l$ and fixed $R$ such that $\tr(R^TR)=1$, let $F_i$ be as in 
(\ref{eq:inteq1ab0b0}). Set
\begin{eqnarray}
   \label{eq:thm1eq1}
 L & = &  
 \max_{X\in\mR^{n\times k},X^TX=I} \sum_{i=1}^{l}  \sqrt{2}\tr(G_i^{(1)} X F_i)   .
\end{eqnarray}
Let $\xi$ be as in (\ref{eq:kronrep6}) and (\ref{eq:inteq1ab0}).  One then has  
\begin{eqnarray}
   \label{eq:thm1eq2}
\lim_{n\rightarrow \infty}  \frac{1}{\sqrt{n}}  \mE \xi   \leq  \lim_{n\rightarrow \infty} \frac{1}{\sqrt{n}} \mE L,
\end{eqnarray}
with the righthand side being the so-called random dual.
\end{theorem}

\begin{proof} For a vector of standard normals $\bar{\g}\in\mR^{l\times 1}$ (independent of all other randomness), we consider the following  centered Gaussian processes indexed by an array $\cX = \{X\}$
  \begin{eqnarray}
\label{eq:mr1}
 \cG (\cX) & \triangleq &  \cG (X)  \triangleq  \sum_{i=1}^{l}  \tr( X^T G_i X F_i ) + \sum_{i=1}^{l} \|F_i\|_F \bar{\g}_i  \nonumber   \\
 \cG_u (\cX) & \triangleq &  \cG_u (X)  \triangleq   \sum_{i=1}^{l}      \sqrt{2}\tr(G_i^{(1)} X F_i) .
  \end{eqnarray}
 For two arrays $\cX^{(1)}=\{ X^{(1)}\}$ and $\cX^{(2)}=\{ X^{(2)}\}$ with $ \lp X^{(i)}\rp^T X^{(i)}=I,i=1,2$,  we further write
  \begin{eqnarray}
\label{eq:mr2}
\mE \cG (\cX^{(1)})\cG (\cX^{(2)})   & =  &  \sum_{i=1}^{l} \lp \phi_i + \|F_i\|_F^2 \rp
\nonumber   \\
\mE \cG_u (\cX^{(1)})\cG_u (\cX^{(2)})  & =  & 2\sum_{i=1}^{l} \theta_i,
  \end{eqnarray}
  where
  \begin{eqnarray}
\label{eq:mr2a0}
 \phi_i & = & 
\mE \tr( \lp X^{(1)}\rp^T G_i X^{(1)} F_i ) 
   \tr(  \lp X^{(2)}  \rp^T G_i X^{(2)} F_i ) 
\nonumber   \\
\theta_i & =  &  \mE \tr(G_i^{(1)}  X^{(1)} F_i)  \tr(G_i^{(1)}  X^{(2)} F_i)  .
  \end{eqnarray}
Moreover,
  \begin{eqnarray}
\label{eq:mr2a1}
 \phi_i & = & 
\mE \tr( \lp X^{(1)}\rp^T G_i X^{(1)} F_i ) 
   \tr(  \lp X^{(2)}  \rp^T G_i X^{(2)} F_i ) 
=
 \tr\lp\lp  X^{(1)} F_i \lp X^{(1)}\rp^T \rp^T \lp  X^{(2)} F_i \lp X^{(2)}\rp^T \rp\rp
   \nonumber \\
   & = & 
 \tr\lp  \lp X^{(2)}\rp^TX^{(1)}  F_i\lp X^{(1)}\rp^TX^{(2)}  F_i   \rp
= \tr\lp  \sqrt{F_i} \lp X^{(2)}\rp^TX^{(1)}  \sqrt{F_i}  \sqrt{F_i} \lp X^{(1)}\rp^TX^{(2)}   \sqrt{F_i}   \rp .\nonumber \\
   \end{eqnarray}
 We also have
 \begin{equation}
\label{eq:mr2a2}
\theta_i  =    \mE \tr(G_i^{(1)}  X^{(1)} F_i)  \tr(G_i^{(1)}  X^{(2)} F_i)  
 =    \tr\lp F_i \lp X^{(1)}\rp^TX^{(2)}  F_i  \rp 
  =     \tr\lp \sqrt{F_i} \lp X^{(1)}\rp^TX^{(2)} \sqrt{F_i} F_i  \rp .
   \end{equation}
    From (\ref{eq:mr2})--(\ref{eq:mr2a2}), one then finds
  \begin{align}
\label{eq:mr5}
  \mE & \cG (\cX^{(1)})\cG (\cX^{(2)})
 -
\mE \cG_u (\cX^{(1)})\cG_u (\cX^{(2)} ) =
\nonumber
\\
 & =      \sum_{i=1}^{l} \lp \phi_i + \|F_i\|_F^2 \rp  -2\sum_{i=1}^{l} \theta_i
\nonumber
\\
& = 
\sum_{i=1}^{l} \lp \tr\lp  \sqrt{F_i} \lp X^{(2)}\rp^TX^{(1)}  \sqrt{F_i}  \sqrt{F_i} \lp X^{(1)}\rp^TX^{(2)}   \sqrt{F_i}   \rp  + \tr(F_iF_i) \rp
\nonumber
\\
& -
  2 \sum_{i=1}^{l}  \tr\lp \sqrt{F_i} \lp X^{(1)}\rp^TX^{(2)} \sqrt{F_i} F_i  \rp 
\nonumber
\\
& = 
\sum_{i=1}^{l} \lp \tr\lp  \sqrt{F_i} \lp X^{(2)}\rp^TX^{(1)}  \sqrt{F_i}  \sqrt{F_i} \lp X^{(1)}\rp^TX^{(2)}   \sqrt{F_i}   \rp  + \tr(F_iF_i) \rp
\nonumber
\\
& -
   \sum_{i=1}^{l}  \tr\lp F_i \sqrt{F_i} \lp X^{(1)}\rp^TX^{(2)} \sqrt{F_i}   \rp 
 -
   \sum_{i=1}^{l}  \tr\lp \sqrt{F_i} \lp X^{(2)}\rp^TX^{(1)} \sqrt{F_i} F_i  \rp 
\nonumber
\\
& = 
   \sum_{i=1}^{l}  \tr
   \lp
   \lp \sqrt{F_i} \lp X^{(1)}\rp^TX^{(2)} \sqrt{F_i} - F_i  \rp^T
   \lp \sqrt{F_i} \lp X^{(1)}\rp^TX^{(2)} \sqrt{F_i} - F_i  \rp
   \rp 
 \nonumber \\
&  \geq   0.  
 \end{align}
We also note  
  \begin{align}
\label{eq:mr5a0}
  \mE & \cG (\cX^{(1)})\cG (\cX^{(1)})
 -
\mE \cG_u (\cX^{(1)})\cG_u (\cX^{(1)} )
  = 
\nonumber
\\
& = 
   \sum_{i=1}^{l}  \tr
   \lp
   \lp \sqrt{F_i} \lp X^{(1)}\rp^TX^{(1)} \sqrt{F_i} - F_i  \rp^T
   \lp \sqrt{F_i} \lp X^{(1)}\rp^TX^{(1)} \sqrt{F_i} - F_i  \rp
   \rp 
\nonumber \\
&  =   0,
  \end{align}
  where the last equality follows since  $\lp X^{(1)}\rp^T X^{(1)}=0$.
  
To complete the proof, we recall on Theorem 1.1 from \cite{Gordon85} (the part of the theorem utilized below was introduced in \cite{Slep62} and is known as Slepian lemma; both results can be deduced as special cases of the concepts discussed in Corollary 3 in \cite{Stojnicgscompyx16}  and in Corollary 4 in \cite{Stojnicgscomp16}).

\begin{theorem}(\cite{Gordon85,Slep62})
\label{thm:Gordonpos1} Let $X_{i}$ and $Y_{i}$, $1\leq i\leq n$, be two centered Gaussian processes which satisfy the following inequalities for all choices of indices
\begin{enumerate}
\item $\mE(X_{i}^2)=\mE(Y_{i}^2)$
\item $\mE(X_{i}X_{l})\leq \mE(Y_{i}Y_{l}), i\neq l$.
\end{enumerate}
 Then
\begin{equation*}
\mE(\min_{i} X_{i})\leq \mE(\min_i Y_{i}) \quad  \Longleftrightarrow \quad \mE(\max_{i} X_{i})\geq \mE(\max_i Y_{i}).
\end{equation*}
\end{theorem}

Applying  Theorem \ref{thm:Gordonpos1} to processes $\cG(\cdot)$ and  $\cG_u(\cdot)$ with correspondence $Y\leftrightarrow\cG$ and $X\leftrightarrow\cG_u$ gives
\begin{align}\label{eq:mt5a1a0}
&  &\mE \max_{\cX^{(a_1)}} \cG(\cX)  & \leq \mE \max_{\cX^{(a_1)}} \cG_u(\cX)
\nonumber \\
\Longleftrightarrow & & \mE \max_{X^TX=I} 
\sum_{i=1}^{l}  \tr( X^T G_i X F_i ) + \sum_{i=1}^{l} \|F_i\|_F \bar{\g}_i 
 & \leq \mE \max_{X^TX=I} \sum_{i=1}^{l}    \sqrt{2}\tr(G_i^{(1)} X F_i) 
\nonumber \\
\Longleftrightarrow & & \mE \max_{X^TX=I} 
\sum_{i=1}^{l}  \tr( X^T G_i X F_i )   
 & \leq \mE \max_{X^TX=I}  \sum_{i=1}^{l}   \sqrt{2}\tr(G_i^{(1)} X F_i) .
 \end{align}
Connecting further (\ref{eq:kronrep6}) and (\ref{eq:mt5a1a0}), we then also find
\begin{eqnarray}\label{eq:mt5a1a1}
 \lim_{n\rightarrow \infty}   \frac{1}{\sqrt{n}}  \mE  \xi  & \leq &   \lim_{n\rightarrow \infty}  \frac{1}{\sqrt{n}} \mE \max_{X^TX=I}  \sum_{i=1}^{l}   \sqrt{2}\tr(G_i^{(1)} X F_i),
 \end{eqnarray}
which, together with (\ref{eq:thm1eq1}), gives  (\ref{eq:thm1eq2}) and completes the proof.
\end{proof}

\begin{remark}
\label{rem:rem0}
To make writing easier and neater, we throughout the paper usually focus on expectations. However, all of the above key quantities trivially concentrate in the dimensional regimes that we consider and the stated results easily extend to hold in probabilistic sense as well. Also, the above results are stated in the $n\rightarrow \infty$ regime. While for this part of presentation that is not necessarily needed, it is stated to ensure agreement with discussions in later sections.
\end{remark}

\subsection{Handling random dual}
\label{sec:handlerd}

Handling the above random dual is relatively simple. We state it for the completeness. The focus on $L$ which is the key object of interest. First, we note 
\begin{eqnarray}
   \label{eq:hrd1}
 L & \triangleq &  
 \max_{X\in\mR^{n\times k}, X^TX=I}  \sum_{i=1}^{l}   \sqrt{2}\tr(G_i^{(1)} X F_i)
   = -  \sqrt{2} \min_{X\in\mR^{n\times k}, X^TX=I}  \sum_{i=1}^{l}  \tr(G_i^{(1)} X F_i) .
\end{eqnarray}
Then for $\Gamma=\Gamma^T$ we write the Lagrangian 
\begin{eqnarray}
   \label{eq:hrd2}
 \cL  = \sum_{i=1}^{l}   \tr(G_i^{(1)} X F_i)  + \tr(\Gamma X^TX )  - \tr(\Gamma)  .
\end{eqnarray}
A combination of (\ref{eq:hrd1}) and  (\ref{eq:hrd2}) together with the strong duality gives
\begin{eqnarray}
   \label{eq:hrd3}
 L  = - \sqrt{2}\min_{X\in\mR^{n\times k}}\max_{\Gamma=\Gamma^T} \cL= -  \sqrt{2} \max_{\Gamma=\Gamma^T} \min_{X\in\mR^{n\times k}} \cL.
\end{eqnarray}
We then  find $X$ derivative
\begin{eqnarray}
   \label{eq:hrd4}
 \frac{d\cL}{dX}  = \sum_{i=1}^{l}    ( F_iG_i^{(1)})^T  + 2X\Gamma .
\end{eqnarray}
After equalling the above derivative to zero, we further have
\begin{eqnarray}
   \label{eq:hrd5}
 X  = -\frac{1}{2}\sum_{i=1}^{l}    ( F_iG_i^{(1)})^T \Gamma^{-1}.
\end{eqnarray}
Plugging this back in (\ref{eq:hrd2}) gives
\begin{eqnarray}
   \label{eq:hrd6}
\min_{X\in\mR^{n\times k}} \cL = -\frac{1}{4}
\tr \lp 
\lp\sum_{i=1}^{l}
( F_iG_i^{(1)})\rp  \lp \sum_{i=1}^{l}( F_iG_i^{(1)})^T \rp \Gamma^{-1}
\rp - \tr(\Gamma).
\end{eqnarray}
Combining (\ref{eq:hrd3}) and  (\ref{eq:hrd6}) further, we obtain  
\begin{eqnarray}
   \label{eq:hrd7}
 L   & = & -  \sqrt{2} \max_{\Gamma=\Gamma^T} \min_{X\in\mR^{n\times k}} \cL
\nonumber \\
& = & \sqrt{2} \min_{\Gamma=\Gamma^T} \lp  \frac{1}{4}
\tr \lp 
\lp\sum_{i=1}^{l}
( F_iG_i^{(1)})\rp  \lp \sum_{i=1}^{l}( F_iG_i^{(1)})^T \rp \Gamma^{-1}
\rp + \tr(\Gamma)\rp.
\end{eqnarray}
The law of large numbers and concentrations give
\begin{eqnarray}
   \label{eq:hrd8}
\lim_{n\rightarrow\infty} \frac{1}{\sqrt{n}} \mE L   
& = & \lim_{n\rightarrow\infty}\frac{\sqrt{2}}{\sqrt{n}} \min_{\Gamma=\Gamma^T} \lp  \frac{n}{4} \tr\lp \lp\sum_{i=1}^{l}
F_i F_i^T\rp \Gamma^{-1}\rp + \tr(\Gamma) \rp.
\end{eqnarray}
After scaling $\Gamma\rightarrow\frac{1}{2}\Gamma\sqrt{n}$ 
\begin{eqnarray}
   \label{eq:hrd9}
\lim_{n\rightarrow\infty} \frac{1}{\sqrt{n}} \mE L   
& = & 
 \frac{1}{\sqrt{2}} \min_{\Gamma=\Gamma^T} \lp \tr\lp\lp \sum_{i=1}^{l}
 F_i F_i^T \rp \Gamma^{-1} \rp+ \tr(\Gamma) \rp
\nonumber \\
& = & 
 \frac{1}{\sqrt{2}} \min_{\Gamma=\Gamma^T} \lp \tr\lp\lp \sum_{i=1}^{l}
( RA_iR^T )( RA_iR^T )^T \rp \Gamma^{-1} \rp+ \tr(\Gamma) \rp
\nonumber \\
& = & 
\min_{\Gamma=\Gamma^T} \bar{L},
\end{eqnarray}
where
\begin{eqnarray}
   \label{eq:hrd9b0}
 \bar{L} \triangleq   \frac{1}{\sqrt{2}} \lp  \tr \lp\lp\sum_{i=1}^{l}
( RA_iR^T )( RA_iR^T )^T \rp \Gamma^{-1} \rp + \tr(\Gamma) \rp
=
\frac{1}{\sqrt{2}} \lp \tr\lp\lp \sum_{i=1}^{l}
F_iF_i^T \rp\Gamma^{-1} \rp + \tr(\Gamma)\rp.
\end{eqnarray}
Combining  (\ref{eq:kronrep5}), (\ref{eq:thm1eq2}), and  (\ref{eq:hrd9}), we find
\begin{eqnarray}\label{eq:hrd9a0}
 \lim_{n\rightarrow \infty}\mE\lambda_n(H) & =  & \max_{\tr(R^TR)=1} \lp \tr(A_0R^TR) +  \sqrt{2}\lim_{n\rightarrow \infty} \frac{1}{\sqrt{n}} \mE \xi \rp  
\nonumber \\
& \leq & 
 \max_{\tr(R^TR)=1} \lp \tr(A_0R^TR) +  \min_{\Gamma=\Gamma^T} \lp \tr\lp\lp \sum_{i=1}^{l}
( RA_iR^T )( RA_iR^T )^T \rp \Gamma^{-1} \rp + \tr(\Gamma) \rp \rp  ,\nonumber \\
\end{eqnarray}
Solving the inner minimization gives optimal $\Gamma$
\begin{eqnarray}\label{eq:hrd9a0a0}
\tilde{\Gamma} = \sqrt{\sum_{i=1}^{l} ( RA_iR^T )( RA_iR^T )^T}
= \sqrt{\sum_{i=1}^{l} F_iF_i^T} , 
\end{eqnarray} 
and
\begin{eqnarray}\label{eq:hrd9a1}
 \lim_{n\rightarrow \infty}\mE\lambda_n(H)  
& \leq & 
 \max_{\tr(R^TR)=1} \lp \tr(A_0R^TR) +  \min_{\Gamma=\Gamma^T} \lp \tr\lp\lp \sum_{i=1}^{l}
( RA_iR^T )( RA_iR^T )^T \rp \Gamma^{-1} \rp + \tr(\Gamma) \rp \rp  
\nonumber \\
&  = & 
 \max_{\tr(R^TR)=1} \lp \tr(A_0R^TR) +  2 \tr \sqrt{ \sum_{i=1}^{l}
( RA_iR^T )( RA_iR^T )^T } \rp  
\nonumber \\
&  = & 
 \max_{\tr(R^TR)=1} \lp \tr(A_0R^TR) +  2\tr\sqrt{ \sum_{i=1}^{l}
 A_iR^TRA_iR^TR } \rp 
 \nonumber \\
&  = & 
 \max_{S=S^T\succeq 0, \tr(S)=1} \lp \tr(A_0S) +  2 \tr\sqrt{ \sum_{i=1}^{l}
 (A_i S)^2 } \rp 
 \nonumber \\
&  = & 
\rho_n ,
\end{eqnarray}
where the last equality is obtained in Proposition 3.3 in \cite{CollYama26} (see, also Theorem 1.9 in \cite{Kunisky26} and \cite{bandfree23,bandfree24} for related results).

Since the above upper bound analysis recovers $\rho_n$, it is essentially tight. However, this is known through spectral characterizations \cite{HaagThor05,CollinsGP22,Schultz05}. Proving it via RDT is quite a challenge. We discuss it in more detail next.

\subsection{Double-checking strong random duality}
\label{sec:strrd}

Random dual  established through Theorem \ref{thm:thm1}  upper-bounds $\mE\xi$. Below we complement this result with a matching lower bound. Considerations from Section \ref{sec:handlerd} will then be sufficient to confirm that not only  $\lim_{n\rightarrow \infty}\mE\lambda_n(H)\leq \rho_n$, but also
$\lim_{n\rightarrow \infty}\mE\lambda_n(H) = \rho_n$.

\subsection{Lower-bounding $\mE \xi$}
\label{sec:strhandbclb}

As mentioned earlier, Theorem \ref{thm:Gordonpos1} can be deduced as a special case of concepts discussed in Corollary 3 in \cite{Stojnicgscompyx16}  and in Corollary 4 in \cite{Stojnicgscomp16}). The machinery developed therein features an important generic principle related to bounding tightness. Namely, it establishes that the bounds are tight provided that two nontrivially overlapped replicated systems cannot double the maximal value of a single system. The very same principle was considered in \cite{TalSph06,TalSK06} for single-partite systems (see, also, e.g., \cite{Stojniccovmat26}). Such systems are of our interest here, but in a significantly more general setting.

In what follows, we check whether this condition indeed holds. Following \cite{Stojnicgscompyx16,Stojnicgscomp16,TalSph06,Stojniccovmat26}, for a real scalar $t\in [0,1]$, we consider basic (non-replicated) interpolated system
 \begin{eqnarray}
 \label{eq:lbstr1}
\mbox{ \textbf{1-rep:} } \quad\quad D(t)=  \max_{X\in\mR^{n\times k},X^TX=I}  \lp \sqrt{t}\sum_{i=1}^{l}  \tr( X^T G_i X F_i )
 +\sqrt{1-t}  \sqrt{2}\tr(G_i^{(1)} X F_i) \rp  .  
   \end{eqnarray}
It is not that difficult to see that $D(t)$ continuously interpolates between $\xi$ and $L$. In particular,  one has
 \begin{eqnarray}
 \label{eq:lbstr2}
 D(1) = \xi \quad\quad \mbox{ and } \quad\quad  
   D(0)= L .
  \end{eqnarray}
Moreover, Theorem \ref{thm:thm1} and machineries of \cite{Stojnicgscompyx16,Stojnicgscomp16,TalSph06,Stojniccovmat26} give
 \begin{eqnarray}
 \label{eq:lbstr2a0}
\lim_{n\rightarrow \infty} \frac{1}{\sqrt{n}}\mE \xi = \lim_{n\rightarrow \infty } \frac{1}{\sqrt{n}}\mE D(1)  \leq \lim_{n\rightarrow \infty } \frac{1}{\sqrt{n}}\mE D(t) \leq  
\lim_{n\rightarrow \infty} \frac{1}{\sqrt{n}}\mE   D(0)= \lim_{n\rightarrow \infty} \frac{1}{\sqrt{n}}\mE   L.
  \end{eqnarray}
 
Let 
 \begin{eqnarray}
 \label{eq:lbstr3a0a0}
\cQ_F 
& \triangleq & \left \{ Q \in\mR^{k\times k} \hspace{.05in} | \hspace{.05in} 
\sum_{i=1}^{l}\tr\lp \lp  \sqrt{F_i} (Q-I)\sqrt{F_i} \rp^T \sqrt{F_i} (Q-I)\sqrt{F_i} \rp >0 \in\mR^{n\times k} \right  \} 
\nonumber \\
& = & 
\left \{ Q \in\mR^{k\times k} \hspace{.05in} | \hspace{.05in} 
\sum_{i=1}^{l}\tr\lp   \sqrt{F_i} (Q-I)^T F_i (Q-I)\sqrt{F_i} \rp >0 \in\mR^{n\times k} \right \}  \nonumber \\
& = & 
\left  \{ Q \in\mR^{k\times k} \hspace{.05in} | \hspace{.05in} 
\sum_{i=1}^{l}\tr\lp     (Q^T-I) F_i (Q-I)F_i \rp >0 \in\mR^{n\times k} \right   \} 
 \end{eqnarray}
and
 \begin{eqnarray}
 \label{eq:lbstr3a0}
\cQ \triangleq \left \{ Q \in\cQ_F \hspace{.05in} | \hspace{.05in} 
\lp X^{(1)}\rp^T X^{(2)} =Q,\lp X^{(1)}\rp^T X^{(1)} =\lp X^{(2)}\rp^T X^{(2)} =I, X^{(1)}, X^{(2)} \in\mR^{n\times k} \right  \} .
 \end{eqnarray}
For $Q\in\cQ$ consider the  set of replica pairs 
 \begin{eqnarray}
 \label{eq:lbstr3a00}
\bar{\cX} = \left \{ (X^{(1)},X^{(2)}) \hspace{.05in} | \hspace{.05in}
X^{(1)},X^{(2)}\in\mR^{n\times k}, \lp X^{(1)}\rp^T X^{(1)}=\lp X^{(2)}\rp^T X^{(2)}=I,  \lp X^{(1)}\rp^T X^{(2)}= Q \right  \} .
  \end{eqnarray}
  We then set 
 \begin{eqnarray}
 \label{eq:lbstr3a1}
 \zeta_i^{(1)}(X^{(j)}) & = &  \tr\lp  \lp X^{(j)}\rp^T G_i X^{(j)} F_i \rp
  \nonumber \\
 \zeta_i^{(2)}(X^{(j)}) & = & \tr(G_i^{(1)}  X^{(j)} F_i) 
 , 
 \end{eqnarray} 
 and associate to $\bar{\cX}$ the following $Q$-overlapped replicated system
 \begin{equation}
 \label{eq:lbstr3}
 \mbox{ \textbf{2-$Q$-rep:} } \hspace{.03in} D^{(2)}(t)=  \max_{(X^{(1)},X^{(2)})\in\bar{\cX}}   \sum_{j=1}^{2} \lp \sqrt{t}\sum_{i=1}^{l} \zeta_i^{(1)} (X^{(j)})
 +\sqrt{1-t}  \sqrt{2} \sum_{i=1}^{l} \zeta_i^{(2)}(X^{(j)})  \rp .  
   \end{equation}
The above principle (which basically states that the \emph{nontrivially overlapped replicated system cannot double the free energy}) practically means that for any $t\in[0,1)$ and $Q\in\cQ$, we must have
 \begin{eqnarray}
 \label{eq:lbstr4}
  \mbox{ \textbf{2-$Q$-rep} } < \mbox{ $\mathbf{2}\times$(\textbf{1-rep}) } .  
   \end{eqnarray}
Switching to mathematical terrain, the above implies that for any $t\in[0,1)$
 \begin{eqnarray}
 \label{eq:lbstr5}
  \mbox{ \textbf{2-$Q$-rep} } < \mbox{ $\mathbf{2}\times$(\textbf{1-rep}) } 
  \quad\quad
\Longleftrightarrow
  \quad\quad
  \min_{Q\in\cQ} \lp 
  \lim_{n\rightarrow \infty} \frac{2}{\sqrt{n}} \mE D(0) 
  -
   \lim_{n\rightarrow \infty} \frac{1}{\sqrt{n}} \mE D^{(2)}(t) 
\rp >0.  
   \end{eqnarray}
Similarly to \cite{Stojniccovmat26}, showing (\ref{eq:lbstr5}) establishes complete $k$-fold matrix analogues to Theorem 2.4 in \cite{TalSK06} and Theorem 5.2 in \cite{TalSph06}. The only additional thing that one has to ensure is that $Q$-overlap is adequately chosen. Following the derivations of 
\cite{Stojnicgscomp16,Stojnicgscompyx16,TalSph06,TalSK06}, the critical overlap is determined  precisely  as the one that ensures nonzero value of the quantity in (\ref{eq:mr5}) (for zero the interpolating contribution is nonexistent). This is effectively subsumed in the definition of $\cQ$ in (\ref{eq:lbstr3a0}). Since $k$ and $l$ are fixed they don't impact concentrations and the remaining parts of the \cite{TalSph06,TalSK06} methodologies automatically extend and ensure $ \lim_{d\rightarrow \infty} \frac{1}{\sqrt{n}}\mE D(1)= \lim_{d\rightarrow \infty} \frac{1}{\sqrt{n}}\mE D(0)$.  Practically speaking, one ultimately has
 \begin{equation}
 \label{eq:lbstr6}
 \min_{Q\in\cQ} \lp 
  \lim_{n\rightarrow \infty} \frac{2}{\sqrt{n}} \mE D(0) 
  -
   \lim_{n\rightarrow \infty} \frac{1}{\sqrt{n}} \mE D^{(2)}(t) 
\rp >0 
\hspace{.04in}\Longleftrightarrow \hspace{.04in}
\lim_{n\rightarrow \infty} \frac{1}{\sqrt{n}} \mE D(1)= \lim_{n\rightarrow \infty} \frac{1}{\sqrt{n}} \mE D(0).  
   \end{equation}

\subsubsection{Establishing $  \mbox{ \textbf{2-$Q$-rep} } < \mbox{ $\mathbf{2}\times$(\textbf{1-rep}) } $}
\label{sec:2rep1rep}

To show that the righthand side of (\ref{eq:lbstr5}) indeed holds (which then implies that $  \mbox{ \textbf{2-$Q$-rep} } < \mbox{ $\mathbf{2}\times$(\textbf{1-rep}) } $ principle is in place as well), we first prove the following  $\mE D^{(2)}(t)$ upper-bounding theorem.

\begin{theorem}
\label{thm:lbstrthm1}
For large $n\in\mN$ and fixed $k,l\in\mN$, let $G_{i}\in\mR^{n\times n}$, $G_{i}^{(1,1)},G_{i}^{(1,2)},G_{i}^{(1)}\in\mR^{k\times n},i=1,\dots,l$ be independent  standard normal matrices. Let $R$, $A_i$,  and $F_i$ be as in  Theorem \ref{thm:thm1}. Also, let $U_iD_iU_i^T=F_i$ be the eigen-decomposition of $F_i$ ($U_i$ is unitary and $D$ is diagonal). For  $t\in(0,1)$ and $Q\in\cQ$, let $D^{(2)}(t)$   be as in (\ref{eq:lbstr3a1}) and  (\ref{eq:lbstr3}). Set $G_i^{(1,3)} = Q G_i^{(1,1)} + \sqrt{I-Q^TQ}G_i^{(1,2)}$, 
 \begin{eqnarray}
 \label{lbstrthm1eq0}
 \bar{\zeta}_i^{(1,1)}(X^{(1)}) & = & \tr(G_i^{(1,1)}  X^{(1)} F_i) 
 \nonumber \\
 \bar{\zeta}_i^{(1,3)}(X^{(2)}) & = & \tr(G_i^{(1,3)}  X^{(2)} F_i) , 
 \end{eqnarray} 
and
\begin{equation}
   \label{eq:lbstrthm1eq1}
\cG_{D_u} (X^{(1)},X^{(2)}) =    
\sqrt{2} \lp
\sqrt{t} \sum_{i=1}^{l}    \bar{\zeta}_i^{(1,1)}(X^{(1)})
+\sqrt{t} \sum_{i=1}^{l}    \bar{\zeta}_i^{(1,3)}(X^{(2)})
+\sqrt{1-t} \sum_{j=1}^{2}  \sum_{i=1}^{l}   \zeta_i^{(2)}(X^{(j)})
\rp .
\end{equation} 
Let
\begin{equation}
   \label{eq:lbstrthm1eq1a0}
L^{(2)}(t)=  \max_{(X^{(1)},X^{(2)})\in\bar{\cX}}  
 \cG_{D_u} (X^{(1)},X^{(2)}) .
\end{equation} 
 One then has  
\begin{eqnarray}
   \label{eq:lbstrthm1eq2}
\lim_{n\rightarrow \infty} \frac{1}{\sqrt{n}} \mE D^{(2)}(t) 
\leq 
\lim_{n\rightarrow \infty} \frac{1}{\sqrt{n}} \mE L^{(2)}(t) .
\end{eqnarray}
\end{theorem}

\begin{proof} We study two centered Gaussian processes indexed by an array $\cX = \{X^{(1)},X^{(2)}\}$
  \begin{align}
\label{eq:lbstrmr1}
 \cG_D (\cX)   \triangleq   \cG_D (X^{(1)},X^{(2)})  
   & =    \sum_{j=1}^{2} \lp \sqrt{t}\sum_{i=1}^{l} \zeta_i^{(1)}(X^{(j)})
 +\sqrt{1-t}  \sqrt{2} \sum_{i=1}^{l} \zeta_i^{(2)}(X^{(j)})  \rp  
\nonumber \\
&  \hspace{.15in} + \sqrt{t}  \sqrt{2\| F_i\|_F^2+ 2\| \sqrt{F_i}Q\sqrt{F_i}\|_F^2}  \bar{\g} 
  \nonumber   \\
 \cG_{D_u} (\cX)  \triangleq   \cG_{D_u} (X^{(1)},X^{(2)})  
 & =
 \sqrt{2} \lp
\sqrt{t} \sum_{i=1}^{l}    \bar{\zeta}_i^{(1,1)} (X^{(1)})
+\sqrt{t} \sum_{i=1}^{l}    \bar{\zeta}_i^{(1,3)} (X^{(2)})
+\sqrt{1-t} \sum_{j=1}^{2}  \sum_{i=1}^{l}    \zeta_i^{(2)}(X^{(j)})
\rp   .\nonumber \\
  \end{align}
Let $\cX^{(a)}=\{ X^{(a_1)}, X^{(a_2)}\}$ and $\cX^{(b)}=\{ X^{(b_1)}, X^{(b_2)}\}$ be two arrays such that 
  \begin{eqnarray}
\label{eq:lbstrmr1a0}
\lp X^{(a_i)} \rp^T X^{(a_i)} =\lp X^{(b_i)} \rp^T X^{(b_i)} =I,i=1,2, \quad  \mbox{and}\quad 
 \lp X^{(a_1)}\rp^T X^{(a_2)}=\lp X^{(b_1)}\rp^T X^{(b_2)}= Q^T .
\end{eqnarray}
Then we have 
  \begin{eqnarray}
\label{eq:lbstrmr2}
\mE \cG (\cX^{(a)})\cG (\cX^{(b)})   & =  &  \sum_{i=1}^{l} \lp t\bar{\phi}_i + 2(1-t)\bar{\phi}_{i,2} 
+ 2t\| F_i\|_F^2+ 2t\| \sqrt{F_i}Q\sqrt{F_i}\|_F^2 \rp
\nonumber   \\
\mE \cG_u (\cX^{(a)})\cG_u (\cX^{(b)})  & =  & 2t\sum_{i=1}^{l} \bar{\theta}_i
+  2(1-t)\sum_{i=1}^{l} \bar{\phi}_{i,2} ,
  \end{eqnarray}
  where
  \begin{eqnarray}
\label{eq:lbstrmr2b0}
 \bar{\phi}_i & = & \mE \lp \zeta_i^{(1)}(X^{(a_1)}) + \zeta_i^{(1)}(X^{(a_2)}) \rp
\lp \zeta_i^{(1)}(X^{(b_1)}) + \zeta_i^{(1)}(X^{(b_2)}) \rp
  \nonumber \\
  \bar{\phi}_{i,2} & = & \mE \lp \zeta_i^{(2)}(X^{(a_1)}) + \zeta_i^{(2)}(X^{(a_2)}) \rp
\lp \zeta_i^{(2)}(X^{(b_1)}) + \zeta_i^{(2)}(X^{(b_2)}) \rp
 \nonumber \\
  \bar{\theta}_{i} & = & \mE \lp \bar{\zeta}_i^{(1,1)}(X^{(a_1)}) + \bar{\zeta}_i^{(1,3)}(X^{(a_2)}) \rp
\lp \bar{\zeta}_i^{(1,1)}(X^{(b_1)}) + \bar{\zeta}_i^{(1,3)}(X^{(b_2)}) \rp . 
 \end{eqnarray}
We then find
  \begin{eqnarray}
\label{eq:lbstrmr2b1}
 \bar{\phi}_i & = &  
 \bar{\phi}_{i,1} +  \bar{\phi}_{i,2} +  \bar{\phi}_{i,3} +  \bar{\phi}_{i,4},
\end{eqnarray}
 where
 \begin{eqnarray}
\label{eq:lbstrmr2b2}
 \bar{\phi}_{i,1} & = & \mE \tr\lp \lp X^{(a_1)}\rp^T G_i X^{(a_1)} F_i \rp 
   \tr\lp \lp X^{(b_1)}  \rp^T G_i X^{(b_1)} F_i \rp 
\nonumber \\
& = & 
\tr\lp  \sqrt{F_i} \lp X^{(b_1)}\rp^TX^{(a_1)}  \sqrt{F_i}  \sqrt{F_i} \lp X^{(a_1)}\rp^TX^{(b_1)}   \sqrt{F_i}   \rp 
 \nonumber \\
 \bar{\phi}_{i,2} & = & \mE \tr\lp \lp X^{(a_2)}\rp^T G_i X^{(a_2)} F_i \rp 
   \tr\lp \lp X^{(b_2)}  \rp^T G_i X^{(b_2)} F_i \rp 
\nonumber \\
& = & 
\tr\lp  \sqrt{F_i} \lp X^{(b_2)}\rp^TX^{(a_2)}  \sqrt{F_i}  \sqrt{F_i} \lp X^{(a_2)}\rp^TX^{(b_2)}   \sqrt{F_i}   \rp 
 \nonumber \\
 \bar{\phi}_{i,3} & = & \mE \tr\lp \lp X^{(a_1)}\rp^T G_i X^{(a_1)} F_i \rp 
   \tr\lp \lp X^{(b_2)}  \rp^T G_i X^{(b_2)} F_i \rp 
\nonumber \\
& = & 
\tr\lp  \sqrt{F_i} \lp X^{(a_1)}\rp^TX^{(b_2)}  \sqrt{F_i}  \sqrt{F_i} \lp X^{(b_2)}\rp^TX^{(a_1)}   \sqrt{F_i}   \rp 
 \nonumber \\
 \bar{\phi}_{i,4} & = & \mE \tr\lp \lp X^{(a_2)}\rp^T G_i X^{(a_2)} F_i \rp 
   \tr\lp \lp X^{(b_1)}  \rp^T G_i X^{(b_1)} F_i \rp
   \nonumber \\
& = & 
\tr\lp  \sqrt{F_i} \lp X^{(b_1)}\rp^TX^{(a_2)}  \sqrt{F_i}  \sqrt{F_i} \lp X^{(a_2)}\rp^TX^{(b_1)}   \sqrt{F_i}   \rp . 
  \end{eqnarray}

After noting that the elements of $G^{(1,3)}$ independent standard normals, we also have
  \begin{eqnarray}
\label{eq:lbstrmr2b3}
 \bar{\theta}_i & = &  
 \bar{\theta}_{i,1} +  \bar{\theta}_{i,2} +  \bar{\theta}_{i,3} +  \bar{\theta}_{i,4},
\end{eqnarray}
 where
  \begin{eqnarray}
\label{eq:lbstrmr2b4}
 \bar{\theta}_{i,1} & =  &  \mE \tr(G_i^{(1,1)}  X^{(a_1)} F_i)  \tr(G_i^{(1,1)}  X^{(b_1)} F_i) 
=\tr\lp \sqrt{F_i} \lp X^{(b_1)}\rp^TX^{(a_1)}  \sqrt{F_i} F_i\rp
\nonumber   \\
\bar{\theta}_{i,2} & =  &  \mE \tr(G_i^{(1,3)}  X^{(a_2)} F_i)  \tr(G_i^{(1,3)}  X^{(b_2)} F_i) 
= \tr\lp \sqrt{F_i} \lp X^{(b_2)}\rp^TX^{(a_2)}  \sqrt{F_i} F_i\rp 
\nonumber   \\
\bar{\theta}_{i,3} & =  &  \mE \tr(G_i^{(1,1)}  X^{(a_1)} F_i)  \tr(G_i^{(1,3)}  X^{(b_2)} F_i) 
= \tr\lp \sqrt{F_i} \lp X^{(a_1)}\rp^TX^{(b_2)}  \sqrt{F_i} \sqrt{F_i} Q \sqrt{F_i}\rp 
\nonumber   \\
\bar{\theta}_{i,2} & =  &  \mE \tr(G_i^{(1,3)}  X^{(a_2)} F_i)  \tr(G_i^{(1,1)}  X^{(b_1)} F_i) 
= \tr\lp \sqrt{F_i} \lp X^{(a_2)}\rp^TX^{(b_1)}  \sqrt{F_i}  \sqrt{F_i} Q^T \sqrt{F_i} \rp .
  \end{eqnarray}
One then observes
  \begin{eqnarray}
\label{eq:lbstrmr2b5}
  \bar{\phi}_{i,1}  -2
 \bar{\theta}_{i,1}  + \tr(F_iF_i)
& = & \| \sqrt{F_i} \lp X^{(b_1)}\rp^TX^{(a_1)}  \sqrt{F_i} -F_i\|_F^2 \geq 0
 \nonumber \\
  \bar{\phi}_{i,2}  -2
 \bar{\theta}_{i,2}  + \tr(F_iF_i)
& = & \| \sqrt{F_i} \lp X^{(b_2)}\rp^TX^{(a_2)}  \sqrt{F_i} -F_i\|_F^2 \geq 0
 \nonumber \\
  \bar{\phi}_{i,3}  -2
 \bar{\theta}_{i,3}  +  \tr \lp (\sqrt{F_i} Q \sqrt{F_i} )^T (\sqrt{F_i} Q \sqrt{F_i} )\rp
& = & \| \sqrt{F_i} \lp X^{(b_2)}\rp^TX^{(a_1)}  \sqrt{F_i} -\sqrt{F_i}Q\sqrt{F_i}\|_F^2 \geq 0
 \nonumber \\
  \bar{\phi}_{i,4}  -2
 \bar{\theta}_{i,4}  +  \tr \lp (\sqrt{F_i} Q \sqrt{F_i} )^T (\sqrt{F_i} Q \sqrt{F_i} )\rp
& = & \| \sqrt{F_i} \lp X^{(b_1)}\rp^TX^{(a_2)}  \sqrt{F_i} -\sqrt{F_i}Q^T\sqrt{F_i}\|_F^2 \geq 0. \nonumber \\
\end{eqnarray}
Moreover,
  \begin{eqnarray}
\label{eq:lbstrmr2b6}
\bar{\phi}_i -2\bar{\theta}_i 
+2\|F_i\|_F^2
+2\|\sqrt{F_i} Q \sqrt{F_i} \|_F^2
& = &  
 \bar{\phi}_{i,1} +  \bar{\phi}_{i,2} +  \bar{\phi}_{i,3} +  \bar{\phi}_{i,4},
-2(
 \bar{\theta}_{i,1} +  \bar{\theta}_{i,2} +  \bar{\theta}_{i,3} +  \bar{\theta}_{i,4}),
\nonumber \\
& & +2\tr(F_iF_i)
+2\tr \lp (\sqrt{F_i} Q \sqrt{F_i} )^T (\sqrt{F_i} Q \sqrt{F_i} )\rp 
\nonumber \\
& \geq & 0.
\end{eqnarray}
Subtracting  second from the first equation in (\ref{eq:lbstrmr2}), we arrive at
   \begin{eqnarray}
\label{eq:lbstrmr2b7}
\mE \cG (\cX^{(a)})\cG (\cX^{(b)}) -
\mE \cG_u (\cX^{(a)})\cG_u (\cX^{(b)}) 
  & =  &  \sum_{i=1}^{l} \lp t\bar{\phi}_i + 2(1-t)\bar{\phi}_{i,2} 
+ 2t\| F_i\|_F^2+ 2t\| \sqrt{F_i}Q\sqrt{F_i}\|_F^2 \rp
\nonumber   \\
 &   & -2t\sum_{i=1}^{l} \bar{\theta}_i
-  2(1-t)\sum_{i=1}^{l} \bar{\phi}_{i,2}
\nonumber \\
& = & t\sum_{i=1}^{l} \lp \bar{\phi}_i   -2\bar{\theta}_i
+ 2\| F_i\|_F^2+ 2\| \sqrt{F_i}Q\sqrt{F_i}\|_F^2 \rp
\nonumber \\
& \geq  &  0 ,
  \end{eqnarray}
where the last inequality follows from (\ref{eq:lbstrmr2b6}). Finally, when $b=a$  all four inequalities in (\ref{eq:lbstrmr2b5}) become equalities (the first two since, by (\ref{eq:lbstrmr1a0}),  $ \lp X^{(a_1)}\rp^TX^{(a_1)}= \lp X^{(a_2)} \rp^TX^{(a_2)}=I$ and the last two since, by (\ref{eq:lbstrmr1a0}),  $ \lp X^{(a_1)}\rp^TX^{(a_2)}=Q$). This then immediately implies equalities in  (\ref{eq:lbstrmr2b6})  and (\ref{eq:lbstrmr2b7}) as well.
 
 After noting  $X\leftrightarrow\cG_{D_u}$ correspondence, the above is then sufficient to apply  Theorem \ref{thm:Gordonpos1}   to processes $\cG_D(\cdot)$ and  $\cG_{D_u}(\cdot)$. After doing so, pne obtains
\begin{align}\label{eq:lbstrmt5a1a0}
&  &\mE \max_{\cX^{(a_1,a_2)}} \cG_D(\cX)  & \leq \mE \max_{\cX^{(a_1,a_2)}} \cG_{D_u}(\cX)
\nonumber \\
\Longleftrightarrow & & \mE \max_{(X^{(1)},X^{(2)})\in\bar{\cX}} \cG_D (X^{(1)},X^{(2)}) & \leq \mE \max_{(X^{(1)},X^{(2)})\in\bar{\cX}} \cG_{D_u} (X^{(1)},X^{(2)}) 
 \nonumber \\
\Longleftrightarrow & & 
\lim_{n\rightarrow \infty}
\frac{1}{\sqrt{n}}\mE D^{(2)}(t)  
+  \lim_{n\rightarrow \infty}
\frac{\sqrt{t}}{\sqrt{n}}\mE \sum_{i=1}^{l} \sqrt{2\| F_i\|_F^2+ 2\| \sqrt{F_i}Q\sqrt{F_i}\|_F^2}  \bar{\g} 
& \leq  \lim_{n\rightarrow \infty}
 \frac{1}{\sqrt{n}}   \mE L^{(2)}(t)
 \nonumber \\
\Longleftrightarrow & & 
\lim_{n\rightarrow \infty}
\frac{1}{\sqrt{n}}\mE D^{(2)}(t)  
 & \leq  \lim_{n\rightarrow \infty}
 \frac{1}{\sqrt{n}}   \mE L^{(2)}(t) \end{align}
which matches (\ref{eq:lbstrthm1eq2}) and completes the theorem's proof. 
\end{proof}

We also observe that since $k$ and $l$ are fixed, quantities on the left and righthand side of the second to last inequality actually concentrate as well.

Recalling on (\ref{eq:lbstr2}), (\ref{eq:lbstr2a0}), and (\ref{eq:lbstrthm1eq2}), we note that the condition
in (\ref{eq:lbstr5}) and  (\ref{eq:lbstr6}) will be met if for any $t<1$
\begin{eqnarray}
   \label{eq:lbstrmt5a1a1}
\lim_{n\rightarrow \infty} \frac{1}{\sqrt{n}} \mE L^{(2)}(t) < 2 \lim_{n\rightarrow \infty}\frac{1}{\sqrt{n}} \mE L(t).
\end{eqnarray}
We characterize $\lim_{n\rightarrow \infty} \frac{1}{\sqrt{n}} \mE L^{(2)}(t) $ next.

\subsubsection{Characterization of $\lim_{n\rightarrow \infty} \frac{1}{\sqrt{n}} \mE L^{(2)}(t) $}
\label{sec:l2}

 Since $L^{(2)}(t)$ is clearly of key importance,  we study it in detail below. To that end, we recall  
 (\ref{eq:lbstrthm1eq1}) and  (\ref{eq:lbstrthm1eq1a0})
\begin{equation}
   \label{eq:l2lbstrthm1eq1}
L^{(2)}(t)= \sqrt{2} \lp
\sqrt{t} \sum_{i=1}^{l}    \bar{\zeta}_i^{(1,1)}(X^{(1)})
+\sqrt{t} \sum_{i=1}^{l}    \bar{\zeta}_i^{(1,3)}(X^{(2)})
+\sqrt{1-t} \sum_{j=1}^{2}  \sum_{i=1}^{l}   \zeta_i^{(2)}(X^{(j)})
\rp .
\end{equation} 
For any $i\in\{1,2,\dots,l\}$ let each of $G_i^{(x,1)}\in\mR^{n\times k}$  and  $G_i^{(x,2)}\in\mR^{n\times k}$ have independent standard normal components. Let $\g_j^{(x,1)}$ be the $j$-th column of $G_i^{(x,1)}$ and analogously $\g_j^{(x,2)}$ be the $j$-th column of $G_i^{(x,2)}$. Let also
\begin{equation}
   \label{eq:l2lbstrthm1eq1a0}
\mE \g_i^{(x,1)}\lp\g_i^{(x,2)}\rp^T = (1-t)I + tQ \triangleq Q_x, 1\leq i\leq n. 
 \end{equation} 
 Also, let pairs $(G_r^{(x,1)}, G_r^{(x,2)})$ and  $(G_s^{(x,1)}, G_s^{(x,2)})$ be independent of each other for any $r\neq s$.
 After setting
  \begin{eqnarray}
   \label{eq:l2lbstrthm1eq1a0a0}
 \tilde{\zeta}_i^{(1,1)}(X^{(1)}) & = & \tr(G_i^{(x,1)}  X^{(1)} F_i) 
 \nonumber \\
 \tilde{\zeta}_i^{(1,3)}(X^{(2)}) & = & \tr(G_i^{(x,2)}  X^{(2)}  F_i) , 
 \end{eqnarray} 
we have the following as a statistical equivalent to (\ref{eq:lbstrthm1eq1}) 
\begin{eqnarray}
   \label{eq:l2lbstrthm1eq1a1}
L^{(2)}(t) 
& = & \sqrt{2} \max_{(X^{(1)},X^{(2)})\in\bar{\cX}}  
\sum_{i=1}^{l} \lp \tilde{\zeta}_i^{(1,1)}(X^{(1)})  + \tilde{\zeta}_i^{(1,3)}(X^{(2)})  \rp
\nonumber \\
& = & \sqrt{2} \max_{(X^{(1)},X^{(2)})\in\bar{\cX}}  
\sum_{i=1}^{l} \lp  \tr(G_i^{(x,1)}  X^{(1)}  F_i)  + \tr(G_i^{(x,2)}  X^{(2)}  F_i) \rp.
\nonumber \\
& = & \sqrt{2} \max_{(X^{(1)},X^{(2)})\in\bar{\cX}}  
\lp  \tr(Z_{1}^T  X^{(1)} )  + \tr( Z_{2}^T   X^{(2)}) \rp,
\end{eqnarray} 
where
\begin{eqnarray}
   \label{eq:l2lbstrthm1eq1a1a0}
Z_{1}^T  =   \sum_{i=1}^{l}  F_i G_i^{(x,1)}
 \quad \mbox{and}\quad 
Z_{2}^T  =   \sum_{i=1}^{l}  F_i G_i^{(x,2)}.
\end{eqnarray}
A combination of Gaussian symmetricity and simple algebraic transformations gives 
that (\ref{eq:l2lbstrthm1eq1a1}) can be rewritten as
\begin{eqnarray}
   \label{eq:l2hrd1}
L^{(2)}(t)= - \sqrt{2} \min_{(X^{(1)},X^{(2)})\in\bar{\cX}}  
\lp  \tr(Z_{1}^T  X^{(1)} )  + \tr( Z_{2}^T   X^{(2)}) \rp.
\end{eqnarray}
We then write the Lagrangian 
\begin{eqnarray}
   \label{eq:l2hrd2}
 \cL & = & \lp  \tr(Z_{1}^T  X^{(1)} )  + \tr( Z_{2}^T   X^{(2)}) \rp
  + \tr\lp \Gamma_1  \lp X^{(1)}\rp^TX^{(1)} \rp - \tr(\Gamma_1 )
 \nonumber \\
 & &  
  + \tr\lp \Gamma_2  \lp X^{(2)}\rp^TX^{(2)}\rp - \tr( \Gamma_2 )
  + \tr\lp \Lambda  \lp X^{(1)}\rp^TX^{(2)}\rp - \tr( \Lambda Q )
.
\end{eqnarray}
Combining  (\ref{eq:l2hrd1}) and  (\ref{eq:l2hrd2}) with duality, one further finds
\begin{eqnarray}
   \label{eq:l2hrd3}
 L^{(2)}(t)  = -\min_{X^{(1)},X^{(2)}}\max_{\Gamma_1=\Gamma_1^T,\Gamma_2=\Gamma_2^T,\Lambda} \cL
 \leq  - \max_{\Gamma_1=\Gamma_1^T,\Gamma_2=\Gamma_2^T,\Lambda}\min_{X^{(1)},X^{(2)}} \cL.
\end{eqnarray}
While it is not necessary due to symmetry, to make writing easier we consider $\Gamma_1=\Gamma_2=\Gamma$ and $\Lambda=\Lambda^T$ which further gives
\begin{eqnarray}
   \label{eq:l2hrd3a0}
 L^{(2)}(t)   
 \leq  - \max_{\Gamma_1=\Gamma_1^T,\Gamma_2=\Gamma_2^T,\Lambda}\min_{X^{(1)},X^{(2)}} \cL
  \leq  - \max_{\Gamma_1=\Gamma_1^T=\Gamma_2=\Gamma,\Lambda=\Lambda^T}\min_{X^{(1)},X^{(2)}} \cL.
\end{eqnarray}
After taking  $X^{(1)}$ and $X^{(2)}$ derivatives, we obtain
\begin{eqnarray}
   \label{eq:l2hrd4}
 \frac{d\cL}{dX^{(1)}} & =  & Z_{1} + 2X_i^{(1)}\Gamma   
 + X^{(2)}\Lambda
 \nonumber \\
 \frac{d\cL}{dX^{(2)}} & =  & Z_{2} + 2X_i^{(2)}\Gamma   
 + X^{(1)}\Lambda
.
\end{eqnarray}
Equalling  derivatives in (\ref{eq:l2hrd4}) to zero gives 
\begin{eqnarray}
   \label{eq:l2hrd4a1}
  Z_{1} + 2X^{(1)}\Gamma   
 + X^{(2)}\Lambda  & = & 0
 \nonumber \\
   Z_{2} + 2X^{(2)}\Gamma   
 + X^{(1)}\Lambda & = & 0
.
\end{eqnarray}
After summing and subtracting the above two equations, we find
\begin{eqnarray}
   \label{eq:l2hrd4a2}
-(Z_{1}+Z_{2}) & = &  (X^{(1)}+X^{(2)})(\Lambda+2\Gamma)
\nonumber \\
-(Z_{1}-Z_{2}) & = & (X^{(1)}-X^{(2)})(-\Lambda+2\Gamma).
\end{eqnarray}
Setting 
\begin{eqnarray}
   \label{eq:l2hrd4a2a0}
W_1 &  =  & \Lambda+2\Gamma
\nonumber \\
W_2 & =  & -\Lambda+2\Gamma,
\end{eqnarray}
allows to write
\begin{eqnarray}
   \label{eq:l2hrd4a2a1}
X^{(1)} + X^{(2)} & = & -(Z_{1}+Z_{2})W_1^{-1} 
\nonumber \\
X^{(1)} - X^{(2)} & = & -(Z_{1}-Z_{2})W_2^{-1} .
\end{eqnarray}
One then easily have
\begin{eqnarray}
   \label{eq:l2hrd4a2a2}
X^{(1)}  & = & \frac{1}{2}(-(Z_{1}+Z_{2})W_1^{-1} - (Z_{1}-Z_{2})W_2^{-1}  )  
\nonumber \\
X^{(2)}  & = &   \frac{1}{2}(-(Z_{1}+Z_{2})W_1^{-1} + (Z_{1}-Z_{2})W_2^{-1}   ) .
 \end{eqnarray}
For the above $X^{(1)}$ and  $X^{(2)}$, we find from (\ref{eq:l2hrd4a1}) 
\begin{eqnarray}
   \label{eq:l2hrd4a2a3}
  Z_{1}^TX^{(1)}& = & -2\Gamma \lp X^{(1)}\rp^T X^{(1)}  
 - \Lambda \lp X^{(2)} \rp^T X^{(1)} 
 \nonumber \\
  Z_{2}^TX^{(2)}& = & -2\Gamma \lp X^{(2)}\rp^T X^{(2)}  
 - \Lambda \lp X^{(1)} \rp^T X^{(2)} .
\end{eqnarray}
A combination of  (\ref{eq:l2hrd2}) and (\ref{eq:l2hrd4a2a3}) then gives
\begin{equation}
   \label{eq:l2hrd4a2a4}
\min_{X^{(1)},X^{(2)}}  \cL 
 = 
  - \tr\lp \Gamma  \lp X^{(1)}\rp^TX^{(1)} \rp 
  - \tr\lp \Gamma  \lp X^{(2)}\rp^TX^{(2)}\rp - 2\tr( \Gamma )
  - \tr\lp \Lambda  \lp X^{(2)}\rp^TX^{(1)}\rp - \tr( \Lambda Q ) ,
\end{equation}
 for $X^{(1)}$ and  $X^{(2)}$ from (\ref{eq:l2hrd4a2a2}). After plugging these $X^{(1)}$ and  $X^{(2)}$ values in (\ref{eq:l2hrd4a2a4}), we obtain
\begin{eqnarray}
   \label{eq:l2hrd4a2a5}
\min_{X^{(1)},X^{(2)}}  \cL 
 & = & 
  - \tr\lp \Gamma  \lp X^{(1)}\rp^TX^{(1)} \rp 
  - \tr\lp \Gamma  \lp X^{(2)}\rp^TX^{(2)}\rp - 2\tr( \Gamma )
  - \tr\lp \Lambda  \lp X^{(2)}\rp^TX^{(1)}\rp - \tr( \Lambda Q ) 
  \nonumber \\
 & = & 
  - \frac{1}{2}\tr\lp \Lambda  \lp X^{(2)}\rp^TX^{(1)}\rp  - \tr\lp \Gamma  \lp X^{(1)}\rp^TX^{(1)} \rp 
  \nonumber \\
 &  & 
  - \frac{1}{2}\tr\lp \Lambda  \lp X^{(2)}\rp^TX^{(1)}\rp   - \tr\lp \Gamma  \lp X^{(2)}\rp^TX^{(2)}\rp - 2\tr( \Gamma )
  - \tr( \Lambda Q ) 
  \nonumber \\
  & = & 
  \tr( Z_{1}^TX^{(1)})   +   \tr( Z_{2}^TX^{(2)})  
  - 2\tr( \Gamma )
  - \tr( \Lambda Q ) 
   \nonumber \\
  & = & 
 -\frac{1}{4} \tr ((Z_{1}+Z_2)^T(Z_1+Z_2)W_1^{-1}) - \frac{1}{4}\tr((Z_1-Z_2)^T(Z_1-Z_2)W_2^{-1})  -2\tr(\Gamma) - \tr(\Lambda Q) .\nonumber \\
\end{eqnarray}
 
The following theorem summarizes the above characterization of $\lim_{n\rightarrow \infty }\frac{1}{\sqrt{n}} \mE L^{(2)}(t)  $.

\begin{theorem}
 \label{thm:lbstrthm2}  
Assume the setup of Theorem \ref{thm:lbstrthm1}. Let $\bar{L}$ be as in (\ref{eq:hrd9b0}). Set 
\begin{eqnarray}
   \label{eq:l2thmchareq1}
Z_{1}^T  =   \sum_{i=1}^{l}  F_i G_i^{(x,1)}  
 \quad \mbox{and}\quad
 Z_{2}^T  =   \sum_{i=1}^{l}  F_i G_i^{(x,2)}   ,
\end{eqnarray}
and
\begin{eqnarray}
   \label{eq:l2thmchareq2}
W_1=\Lambda+2\Gamma \quad \mbox{and}\quad 
W_2= -\Lambda+2\Gamma.
\end{eqnarray}
Let
\begin{equation}
   \label{eq:lbstrthm2eq1a0}
\bar{L}^{(2,G)} \triangleq
\sqrt{2} \lp  \frac{1}{4} \tr ((Z_{1}+Z_2)^T(Z_1+Z_2)W_1^{-1}) + \frac{1}{4}\tr((Z_1-Z_2)^T(Z_1-Z_2)W_2^{-1})  +2\tr(\Gamma) + \tr(\Lambda Q) \rp.
\end{equation}
  One then has  
\begin{eqnarray}
   \label{eq:lbstrthm2eq2}
\lim_{n\rightarrow \infty } \frac{1}{\sqrt{n}} \mE D^{(2)}(t) 
 \leq  
\lim_{n\rightarrow \infty } \frac{1}{\sqrt{n}} \mE L^{(2)}(t) 
 \leq  
\lim_{n\rightarrow \infty }  \frac{1}{\sqrt{n}} \mE  \min_{\Gamma=\Gamma^T,\Lambda=\Lambda^T}  \bar{L}^{(2,G)} .
\end{eqnarray}
\end{theorem}

\begin{proof}
  Follows from the above discussion.
\end{proof}

\begin{corollary}
  \label{thm:lbstrthm3}  
  Assume the setup of Theorem \ref{thm:lbstrthm1}. Let $\bar{L}$ be as in (\ref{eq:hrd9b0}). Set 
 \begin{eqnarray}
   \label{eq:l2thmchareq2}
   \label{eq:coreq1}
K  =   \sum_{i=1}^{l}  F_iF_i^T =   \sum_{i=1}^{l}  F_i^TF_i 
 \quad \mbox{and}\quad   
K_Q  =   \sum_{i=1}^{l}  F_iQ_xF_i^T
 \quad \mbox{and}\quad   
 W_1=\Lambda+2\Gamma \quad \mbox{and}\quad 
W_2= -\Lambda+2\Gamma.
\end{eqnarray}
Let
\begin{equation}
  \label{eq:lbstrthm2eq1}
  \bar{L}^{(2)} \triangleq
 \sqrt{2} \lp \frac{1}{4} \tr ((2K  + K_Q + K_Q^T )W_1^{-1}) + \frac{1}{4}\tr((2K  - K_Q - K_Q^T )W_2^{-1})  +2\tr(\Gamma) + \tr(\Lambda Q) \rp.
\end{equation}
  One then has  
\begin{eqnarray}
   \label{eq:lbstrthm2eq2}
\lim_{n\rightarrow \infty } \frac{1}{\sqrt{n}} \mE D^{(2)}(t) 
 \leq  
\lim_{n\rightarrow \infty } \frac{1}{\sqrt{n}} \mE L^{(2)}(t) 
 \leq  
  \min_{\Gamma=\Gamma^T,\Lambda=\Lambda^T}  \bar{L}^{(2)} .
\end{eqnarray}
Furthermore, if
\begin{eqnarray}
   \label{eq:lbstrthm2eq3}
   \max_{t\in (0,1),Q\in\cQ, Q^TQ\neq I}
 \min_{\Gamma=\Gamma^T,\Lambda=\Lambda^T}  \bar{L}^{(2)}  
 < 2\lim_{n\rightarrow \infty } \min_{\Gamma= \Gamma^T}  \bar{L}   ,
\end{eqnarray}
then the condition in (\ref{eq:lbstr5}) is met. This also implies 
\begin{eqnarray}
   \label{eq:lbstrthm2eq4} 
\lim_{n\rightarrow \infty } \frac{1}{\sqrt{n}} \mE D(1)= \lim_{n\rightarrow \infty } \frac{1}{\sqrt{n}} \mE D(0),
\end{eqnarray}
and based on (\ref{eq:lbstr2}) and (\ref{eq:lbstr2a0})
  \begin{eqnarray}
   \label{eq:lbstrthm2eq5} 
 \lim_{n\rightarrow \infty } \frac{1}{\sqrt{n}}\mE \xi  = \lim_{n\rightarrow \infty } \frac{1}{\sqrt{n}} \mE L.
  \end{eqnarray}
\end{corollary} 

\begin{proof}
  The first part, i.e.,  (\ref{eq:lbstrthm2eq2}), follows from Theorem  \ref{thm:lbstrthm2} (in particular, equations (\ref{eq:l2thmchareq1})--(\ref{eq:lbstrthm2eq1a0})) after one recognizes that the law of large numbers, concentrations, chosen dimensional setup (with large $n$ and fixed $k$ and $l$), and cosmetic scalings $\Gamma\rightarrow \sqrt{n}\Gamma$ and $\Lambda\rightarrow \sqrt{n}\Lambda$ first give
  \begin{eqnarray}
   \label{eq:prcoreq1}
 \mE \bar{L}^{(2,G)} 
& = &
   \mE \lp \frac{1}{4} \tr ((Z_{1}+Z_2)^T(Z_1+Z_2)W_1^{-1}) + \frac{1}{4}\tr((Z_1-Z_2)^T(Z_1-Z_2)W_2^{-1})  +2\tr(\Gamma) + \tr(\Lambda Q) \rp
 \nonumber \\
& \rightarrow &   \frac{n}{4\sqrt{n}} \tr ((2K  + K_Q + K_Q^T )W_1^{-1}) + \frac{1}{4}\tr((2K  - K_Q - K_Q^T )W_2^{-1}) 
   +2\sqrt{n}\tr(\Gamma) + 2\sqrt{n}\tr(\Lambda Q) , \nonumber \\
\end{eqnarray}
  and then
  \begin{eqnarray}
   \label{eq:prcoreq1}
\lim_{n\rightarrow \infty} \frac{1}{\sqrt{n}}\mE \bar{L}^{(2,G)} = \bar{L}^{(2)}, 
  \end{eqnarray}
  The second part and equations (\ref{eq:lbstrthm2eq3})-(\ref{eq:lbstrthm2eq5}) follow from   (\ref{eq:hrd8}), (\ref{eq:lbstr2}), (\ref{eq:lbstr2a0}), and (\ref{eq:lbstr5}).  
\end{proof}

\subsubsection{Verifying condition (\ref{eq:lbstrthm2eq3})}
\label{sec:chkcond}

Theorem  \ref{thm:lbstrthm2} allows to numerically check whether (\ref{eq:lbstrthm2eq3}) holds. That would in principle complete the proof up to the level of numerical precision. However, even though $k$ and $l$ are fixed, they can still be large which can make the numerical computations a bit challenging. The following theorem analytically confirms that condition (\ref{eq:lbstrthm2eq3}) is met.
\begin{theorem}
 \label{thm:lbstrthm3}  
Under the setup of Theorems \ref{thm:lbstrthm1} and \ref{thm:lbstrthm2}, (\ref{eq:lbstrthm2eq3}) holds.
\end{theorem}

\begin{proof}
The proof is based on contradiction. We therefore start by assuming the opposite, i.e., that (\ref{eq:lbstrthm2eq3}) does not hold. From (\ref{eq:hrd9})  and (\ref{eq:lbstrthm2eq1}), one then has that there must be some $\hat{\Gamma}_x$, $\hat{\Lambda}_x$, and $\tilde{\Gamma}_x$ such that 
\begin{eqnarray}
   \label{eq:lemprf0}
(\hat{\Gamma},\hat{\Lambda}) = \mbox{arg}\min_{\Gamma=\Gamma^T,\Lambda=\Lambda^T}\bar{L}^{(2)} \quad\quad \mbox{and}\quad\quad
\tilde{\Gamma}   = \mbox{arg}\min_{\Gamma = \Gamma^T }\bar{L},
\end{eqnarray}
and  
\begin{align}
   \label{eq:lemprf1}
  \sqrt{2}     \min_{\Gamma=\Gamma^T,\Lambda=\Lambda^T}
       & 
       \lp \frac{1}{4}
           \tr ((2K  + K_Q + K_Q^T )W_1^{-1}) + \frac{1}{4}\tr((2K  - K_Q - K_Q^T )W_2^{-1})  +2\tr(\Gamma) + \tr(\Lambda Q) \rp
    \nonumber 
    \\ 
    & =  
        \min_{\Gamma=\Gamma^T,\Lambda=\Lambda^T}\bar{L}^{(2)} 
       = 
       2 \min_{\Gamma=\Gamma^T} \bar{L} 
     \nonumber 
    \\ 
    & = 
       \frac{2}{\sqrt{2}} \min_{\Gamma=\Gamma^T} \lp    \tr\lp\lp     \sum_{i=1}^{l}
F_iF_i^T \rp \Gamma^{-1}\rp + \tr(\Gamma)\rp 
 \nonumber 
    \\ 
    & = 
        \sqrt{2} \min_{\Gamma=\Gamma^T} \lp    \tr\lp K \Gamma^{-1}\rp + \tr(\Gamma)\rp  .
\end{align}
Moreover,
\begin{eqnarray}
   \label{eq:lemprf2}
 \left . \frac{d\bar{L}^{(2)}}{d\Gamma} \right |_{(\Gamma,\Lambda)=( \hat{\Gamma},\hat{\Lambda})}
 =
 \left . \frac{d\bar{L}^{(2)}}{d\Lambda} \right |_{(\Gamma,\Lambda)=( \hat{\Gamma},\hat{\Lambda})}
 =\left . \frac{d\bar{L}}{d\Gamma} \right |_{\Gamma= \tilde{\Gamma}}= 0 .
 \end{eqnarray}
One first observes that the choice $(\hat{\Gamma},\hat{\Lambda})=\lp\frac{\tilde{\Gamma}}{2},0\rp$ satisfies (\ref{eq:lemprf1}) since
\begin{align}
   \label{eq:lemprf3}
    \sqrt{2} & \lp  \frac{1}{4}
           \tr ((2K  + K_Q + K_Q^T )(\hat{\Lambda} +2\hat{\Gamma})^{-1}) + \frac{1}{4}\tr((2K  - K_Q - K_Q^T )(-\hat{\Lambda} +2\hat{\Gamma})^{-1})  +2\tr(\hat{\Gamma}) + \tr(\hat{\Lambda} Q)  \rp
\nonumber \\
 & =  
\sqrt{2} \lp \frac{1}{4}
           \tr ((2K  + K_Q + K_Q^T )(\tilde{\Gamma})^{-1}) + \frac{1}{4}\tr((2K  - K_Q - K_Q^T )(\tilde{\Gamma})^{-1})  +\tr(\tilde{\Gamma})    \rp
 \nonumber \\
 & = 
 \sqrt{2} \lp
           \tr (K   (\tilde{\Gamma})^{-1})   +\tr(\tilde{\Gamma})    \rp .
\end{align}
We then set
\begin{eqnarray}
  \label{eq:lamdereq1}
  B_1 & = & 2K  + K_Q + K_Q^T 
  \nonumber \\
  B_2 & = & 2K  - K_Q - K_Q^T 
  \nonumber \\
  \bar{L}^{(2)} & = &
 \sqrt{2} \lp \frac{1}{4} \tr (B_1W_1^{-1}) + \frac{1}{4}\tr(B_2W_2^{-1})  +2\tr(\Gamma) + \tr(\Lambda Q) \rp,
\end{eqnarray}
and write for the $\Gamma$ derivative
 \begin{eqnarray}
   \label{eq:lemprf5}
  \frac{d\bar{L}^{(2)}}{d\Gamma} 
  & = &
 \sqrt{2} \lp -\frac{1}{4} W_1^{-1}(B_1 +B_1^T ) W_1^{-1} - \frac{1}{4} W_2^{-1}(B_2 +B_2^T ) W_2^{-1}   +   2I\rp.
  \end{eqnarray}
Evaluating for $(\Gamma,\Lambda)=\lp\frac{\tilde{\Gamma}}{2},0\rp$ gives
 \begin{eqnarray}
   \label{eq:lemprf6}
\left .  \frac{d\bar{L}^{(2)}}{d\Gamma} \right |_{(\Gamma,\Lambda)=\lp\frac{\tilde{\Gamma}}{2},0\rp}
    & = & 
\sqrt{2} \lp -\frac{1}{2} \tilde{\Gamma}^{-1}(B_1 +B_1^T +B_2 +B_2^T ) \tilde{\Gamma}^{-1}    + 2I \rp    
 \nonumber \\
 & = &
\sqrt{2} \lp -\frac{1}{2} \tilde{\Gamma}^{-1}(4K ) \tilde{\Gamma}^{-1}    +  2I \rp    
 \nonumber \\
 & = &
\sqrt{2} \lp - 2\tilde{\Gamma}^{-1}K  \tilde{\Gamma}^{-1}    +  2I \rp    
 \nonumber \\
 & = &
0,
  \end{eqnarray}
where the last equality follows since $\tilde{\Gamma}=\sqrt{K}$ from (\ref{eq:hrd9a0a0}). We then find for the $\Lambda$ derivative
 \begin{eqnarray}
   \label{eq:lemprf7}
  \frac{d\bar{L}^{(2)}}{d\Lambda} 
  & = &
 \sqrt{2} \lp -\frac{1}{8} W_1^{-1}(B_1 +B_1^T ) W_1^{-1} + \frac{1}{8} W_2^{-1}(B_2 +B_2^T ) W_2^{-1}   +  \frac{1}{2}(Q^T + Q) \rp.
  \end{eqnarray}
Evaluating for $(\Gamma,\Lambda)=\lp\frac{\tilde{\Gamma}}{2},0\rp$  gives
 \begin{eqnarray}
   \label{eq:lemprf8}
\left .  \frac{d\bar{L}^{(2)}}{d\Lambda} \right |_{(\Gamma,\Lambda)=\lp\frac{\tilde{\Gamma}}{2},0\rp}
    & = & 
\sqrt{2} \lp -\frac{1}{8} \tilde{\Gamma}^{-1}(B_1 +B_1^T -B_2 -B_2^T ) \tilde{\Gamma}^{-1}    +  \frac{1}{2}(Q^T + Q) \rp    
 \nonumber \\
 & = &
\sqrt{2} \lp -\frac{1}{2} \tilde{\Gamma}^{-1}(K_Q+K_Q^T ) \tilde{\Gamma}^{-1}    +  \frac{1}{2}(Q^T + Q) \rp    
 \nonumber \\
 & = &
\frac{\sqrt{2}}{2} \lp - \tilde{\Gamma}^{-1}(K_Q+K_Q^T ) \tilde{\Gamma}^{-1}    +  (Q^T + Q) \rp    
 \nonumber \\
 & = &
\frac{\sqrt{2}}{2} \lp  - \sqrt{K}^{-1} \sum_{i=1}^{l}F_i ( Q_x+Q_x^T ) F_i^T  \sqrt{K}^{-1}    +  (Q^T + Q) \rp    
 \nonumber \\
 & = &
\frac{\sqrt{2}}{2} \lp   -\sqrt{K}^{-1} \sum_{i=1}^{l}F_i (  2(1-t)I + t(Q+Q^T)  ) F_i^T  \sqrt{K}^{-1}    +  (Q^T + Q) \rp . 
  \end{eqnarray}
If $(\Gamma,\Lambda)=\lp\frac{\tilde{\Gamma}}{2},0\rp$ is indeed a stationary point, then one must have
 \begin{eqnarray}
   \label{eq:lemprf9}
\left .  \frac{d\bar{L}^{(2)}}{d\Lambda} \right |_{(\Gamma,\Lambda)=\lp\frac{\tilde{\Gamma}}{2},0\rp}
  & = &
 \lp   -\sqrt{K}^{-1} \sum_{i=1}^{l}F_i (  2(1-t)I + t(Q+Q^T)  ) F_i^T  \sqrt{K}^{-1}    +  (Q^T + Q) \rp = 0. 
  \end{eqnarray}
Moreover, the following sequence of implications must hold as well
\begin{align}
   \label{eq:lemprf10}
   & &
     -\sqrt{K}^{-1} \sum_{i=1}^{l}F_i (  2(1-t)I + t(Q+Q^T)  ) F_i^T  \sqrt{K}^{-1}    +  (Q^T + Q) \ & = 0
 \nonumber \\
 \Longrightarrow & &
    - \sum_{i=1}^{l}F_i (  2(1-t)I + t(Q+Q^T)  ) F_i^T     +  \sqrt{K}(Q^T + Q) \sqrt{K}  & = 0
 \nonumber \\
 \Longrightarrow & &
   \tr\lp  - \sum_{i=1}^{l}F_i (  2(1-t)I + t(Q+Q^T)  ) F_i^T  \rp   + \tr\lp \sqrt{K}(Q^T + Q) \sqrt{K} \rp & = 0
 \nonumber \\
 \Longrightarrow & &
   \tr\lp  - \sum_{i=1}^{l} (  2(1-t)I + t(Q+Q^T)  ) F_iF_i  \rp   + \tr\lp  (Q^T + Q) K \rp & = 0
 \nonumber \\
 \Longrightarrow & &
   \tr\lp  - \sum_{i=1}^{l} (  2(1-t)I + t(Q+Q^T)  ) F_iF_i  \rp   + \tr\lp  (Q^T + Q) \sum_{i=1}^{l} F_iF_i \rp & = 0
 \nonumber \\
 \Longrightarrow & &
   \tr\lp  - \sum_{i=1}^{l} (  2(1-t)I + t(Q+Q^T)  ) F_iF_i  \rp   + \tr\lp \sum_{i=1}^{l} (Q^T + Q)  F_iF_i \rp & = 0
 \nonumber \\
 \Longrightarrow & &
   \tr\lp  - \sum_{i=1}^{l} (  2(1-t)I + t(Q+Q^T)  -(Q+Q^T) ) F_iF_i  \rp    & = 0
 \nonumber \\
 \Longrightarrow & &
  (1-t) \tr\lp  - \sum_{i=1}^{l} ( 2I -(Q+Q^T) ) F_iF_i  \rp    & = 0
 \nonumber \\
 \Longrightarrow & &
  (1-t) \tr\lp  - \sum_{i=1}^{l} ( (I-Q^TQ) + (I -(Q+Q^T) +Q^TQ ) ) F_iF_i  \rp    & = 0
 \nonumber \\
 \Longrightarrow & &
  (1-t) \tr\lp  - \sum_{i=1}^{l} ( I - Q^TQ ) F_iF_i  \rp +
  (1-t) \tr\lp  - \sum_{i=1}^{l} ( I - Q^T)( I - Q)  F_iF_i  \rp    & = 0
 \nonumber \\
 \Longrightarrow & &
 - (1-t) \tr\lp   \sum_{i=1}^{l} ( I - Q^TQ ) F_iF_i  \rp 
  - (1-t) \tr\lp   \sum_{i=1}^{l} \|( I - Q)  F_i\|_F^2 \rp    & = 0
  . 
  \end{align}
Recalling on the definition of $Q_F$ from (\ref{eq:lbstr3a0a0}), we observe that for $X^{(1)}\in\mR^{n\times l}$ and $X^{(1)}\in\mR^{n\times l}$ such that $\lp X^{(1)}\rp^TX^{(1)}=\lp X^{(2)}\rp^TX^{(2)}$, and $\lp X^{(1)}\rp^TX^{(2)}=Q$, one has
\begin{eqnarray}
   \label{eq:lemprf10a0}
 I - Q^TQ & = & I -  \lp \lp X^{(1)}\rp^TX^{(2)} \rp^T \lp X^{(1)}\rp^TX^{(2)} 
\nonumber \\
  & = & I -   \lp X^{(2)} \rp^T  X^{(1)} \lp X^{(1)}\rp^T   X^{(2)}  
  \nonumber \\
  & \succeq & I -   \lp X^{(2)} \rp^TX^{(2)}  
    \nonumber \\
  & = & 0, 
  \end{eqnarray}
  where positive semi-definiteness in the third line follows since $I\succeq X^{(1)} \lp X^{(1)}\rp^T$ as $ X^{(1)} \lp X^{(1)}\rp^T$ is a projection matrix.  From (\ref{eq:lemprf10a1}) we then also have
\begin{eqnarray}
   \label{eq:lemprf10a1}  
    \tr\lp   \sum_{i=1}^{l} ( I - Q^TQ ) F_iF_i  \rp 
    =    \tr\lp   \sum_{i=1}^{l} F_i ( I - Q^TQ ) F_i  \rp  \geq 0.
\end{eqnarray}
Utilizing Cauchy-Schwartz inequality, we further find
\begin{equation}
   \label{eq:lemprf11}
   \tr\lp  ( I - Q^T)F_i ( I - Q) F_i  \rp   \leq   
  \|( I - Q^T)  F_i\|_F \|( I - Q)  F_i\|_F. 
  \end{equation}
Since $Q\in\cQ$, where $\cQ$ is as defined in (\ref{eq:lbstr3a0a0}) and (\ref{eq:lbstr3a0}), one also has
\begin{equation}
   \label{eq:lemprf12}
  0 < \tr\lp  ( I - Q^T)F_i ( I - Q) F_i  \rp   \leq   
  \|( I - Q^T)  F_i\|_F \|( I - Q)  F_i\|_F. 
  \end{equation}
A combination of  (\ref{eq:lemprf10a1}) and   (\ref{eq:lemprf12}) gives that for $t\in(0,1)$ and $Q\in\cQ$
\begin{equation}
   \label{eq:lemprf14}
  - (1-t) \tr\lp   \sum_{i=1}^{l} ( I - Q^TQ ) F_iF_i  \rp 
  - (1-t) \tr\lp   \sum_{i=1}^{l} \|( I - Q)  F_i\|_F^2 \rp     < 0 ,
  \end{equation}
  This violates the ending equality in (\ref{eq:lemprf10}),
which means that the first equality in (\ref{eq:lemprf10})  is incorrect. In other words, one must have 
\begin{equation}
   \label{eq:lemprf15}
     -\sqrt{K}^{-1} \sum_{i=1}^{l}F_i (  2(1-t)I + t(Q+Q^T)  ) F_i^T  \sqrt{K}^{-1}    +  (Q^T + Q)  \neq 0 , 
  \end{equation}
which based on (\ref{eq:lemprf9}) further implies that 
 \begin{eqnarray}
   \label{eq:lemprf16}
\left .  \frac{d\bar{L}^{(2)}}{d\Lambda} \right |_{(\Gamma,\Lambda)=\lp\frac{\tilde{\Gamma}}{2},0\rp}
  & = &
 \lp   -\sqrt{K}^{-1} \sum_{i=1}^{l}F_i (  2(1-t)I + t(Q+Q^T)  ) F_i^T  \sqrt{K}^{-1}    +  (Q^T + Q) \rp \neq 0. 
  \end{eqnarray}
This also means that for $t\in(0,1)$ and $Q\in\cQ$, ${(\Gamma,\Lambda)=\lp\frac{\tilde{\Gamma}}{2},0\rp}$ can not be a stationary point. This, on the other hand, violates the starting assumption from  (\ref{eq:lemprf0})--(\ref{eq:lemprf2}), implies  
 \begin{eqnarray}
   \label{eq:lemprf17}
        \min_{\Gamma=\Gamma^T,\Lambda=\Lambda^T}\bar{L}^{(2)} 
       < 
       2 \min_{\Gamma=\Gamma^T} \bar{L}, 
\end{eqnarray}
and ultimately completes the proof.

\end{proof}

\section{Conclusion}
\label{sec:conc}

In \cite{Lehner99}, Lehner established a deterministic formula to determine the spectral edges of Kronecker-Gaussian matrices, providing a practical mechanism to handle the semicircular counterpart associated through asymptotic freeness with traditional Gaussian ensembles.

We revisit this classical result and develop a methodology to reprove it. Pursuing an avenue different from traditional random matrix theory spectral methods, we reconfirm Lehner's formula by relying on concepts utilized within \emph{Random Duality Theory} (RDT) \cite{StojnicCSetam09,StojnicRegRndDlt10,Stojniccovmat26}. This also automatically reestablishes the key results associated with strong asymptotic freeness proven in \cite{Schultz05,HaagThor05} via spectral methods.

The developed mechanisms are generic and allow for further extensions. To keep the presentation concise, we focused on the results obtained when positive-definiteness of $A_i$'s (the Kronecker product deterministic part) is present. Proving the indefinite counterpart will be the subject of future work.

\section*{Acknowledgment}

The author would like to thank Ramon van Handel for a fruitful discussion on several related topics and in particular for pointing out the relevance of the problems studied here.

%
%
%
%
%
%
%

\begin{singlespace}
\bibliographystyle{plain}
\bibliography{nflgscompyxRefs}
\end{singlespace}

\end{document}